\documentclass[aop,preprint]{imsart}

\usepackage[latin2]{inputenc}

\RequirePackage[OT1]{fontenc}
\RequirePackage{amsthm,amsmath,natbib}
\RequirePackage[colorlinks]{hyperref}
\RequirePackage{hypernat}

\RequirePackage[OT1]{fontenc}
\RequirePackage{amsthm,amsmath,natbib}
\RequirePackage[colorlinks]{hyperref}
\RequirePackage{hypernat}

\numberwithin{equation}{section}
\usepackage{amssymb}
\usepackage{comment}
\usepackage{appendix}
\usepackage{graphicx}
\usepackage{subfigure}
\usepackage{enumerate}
\usepackage{pdfsync}
\usepackage{tikz}
\usepackage{pgffor}

\newtheorem{theorem}{Theorem}[section]
\newtheorem{lemma}[theorem]{Lemma}
\newtheorem{proposition}[theorem]{Proposition}
\newtheorem{corollary}[theorem]{Corollary}

\newtheorem{definition}[theorem]{Definition}
\newtheorem{assumption}[theorem]{Assumption}
\newtheorem{remark}[theorem]{Remark}

\newcommand{\Ev}{\mathbf{E}}

\newcommand{\Pv}{\mathbf{P}}

\newcommand{\CA}{\mathcal {A}}

\newcommand{\CD}{\mathcal {D}}

\newcommand{\CF}{\mathcal {F}}

\newcommand{\CP}{\mathcal {P}}

\newcommand{\CS}{\mathcal {S}}

\newcommand{\CH}{\mathcal {H}}

\newcommand*{\la}{\lambda}

\newcommand*{\Zb}{\mathbb Z}

\newcommand*{\un}[1]{\underline{#1}}

\newcommand*{\be}{\begin{equation}}
\newcommand*{\ee}{\end{equation}}
\newcommand*{\ba}{\begin{aligned}}
\newcommand*{\ea}{\end{aligned}}
\newcommand*{\barr}{\begin{array}{c}}
\newcommand*{\earr}{\end{array}}

\newcommand*{\ind}{\mathbf{1}}

\newcommand*{\wla}{\widetilde \lambda}
\newcommand*{\hla}{\hat \lambda}
\newcommand*{\Th}{\mathrm{th}}

\def \ind     {1\!\!1}
\newcommand{\indd}[1]   {\ensuremath{ \ind \left[ #1\right]}}

\newcommand{\prob}[1]    {\ensuremath{\mathbf{P}\left[#1\right]}}
\newcommand{\expect}[1]  {\ensuremath{\mathbf{E}\left[#1\right]}}

\newcommand{\condprob}[2]    {\ensuremath{\mathbf{P}\left[#1\,\big|\,#2\right]}}
\newcommand{\condexpect}[2]  {\ensuremath{\mathbf{E}\left[#1\,\big|\,#2\right]}}

\def \PP    {\mathbf{P}}
\def \toinp    {\buildrel {\PP}\over{\longrightarrow}}
\def \toindis  {\buildrel {d}\over{\longrightarrow}}
\def \toas     {\buildrel {a.s.}\over{\longrightarrow}}

\begin{document}
\begin{frontmatter}
\title{First Passage Percolation on Inhomogeneous Random Graphs}
\runtitle{FPP on Inhomogeneous RG-s}

\begin{aug}

\runauthor{I. Kolossv\'ary and J. Komj\'athy }

\author{Istv\'an Kolossv\'ary \thanks{This research project was supported by the TÁMOP-4.2.2.C-11/1/KONV-2012-0001**project, supported by the European Union,
co-financed by the European Social Fund}}
\address{Budapest University of Technology and Economics
\tt{istvanko@math.bme.hu}}

\author{J\'ulia Komj\'athy \thanks{This research project was supported by the grant
 KTIA-OTKA  $\#$ CNK 77778, funded by the Hungarian National Development Agency (NFÜ)  from a
source provided by KTIA.}}
\address{
\tt{komyju@math.bme.hu}}

\affiliation{Budapest University of Technology and Economics \\
 Inter-University Centre for Telecommunications and Informatics 4028 Debrecen, Kassai út 26.}

\end{aug}
\date{\today}

\begin{abstract}
 We investigate first passage percolation on inhomogeneous random graphs. The random graph model $G(n,\kappa)$ we study is the model introduced by Bollobás, Janson and Riordan in \cite{BB_IHRGM}, where each vertex has a type from a type space $\mathcal{S}$ and edge probabilities are independent, but depending on the types of the end vertices. Each edge is given an independent exponential weight. We determine the distribution of the weight of the shortest path between uniformly chosen vertices in the giant component and show that the hopcount, i.e.\, the number of edges on this minimal weight path, properly normalized follows a central limit theorem.
We handle the cases where $\wla_n\to\wla$ is finite or infinite, under the assumption that the average number of neighbors $\wla_n$ of a vertex is independent of the type. The paper is a generalization of \cite{RvdH_FPP} written by Bhamidi, van der Hofstad and Hooghiemstra, where FPP is explored on the Erd\H os-R\'enyi graphs.
\end{abstract}



\end{frontmatter}

\maketitle


\section{Introduction and the main results}\label{sec::intro&results}

First passage percolation (FPP), generally speaking, deals with the asymptotic behavior of first-passage times of percolating fluid in some random environment. This topic has gained much attention due to its application in various fields such as interacting particle systems, statistical physics, epidemic models and real-world networks, just to name a few.

Particularly, two quantities of interest for FPP on finite weighted random graph models are the minimal weight of a path between two vertices $x$ and $y$ and the number of edges, often referred to the hopcount, on this path. Without edge weights these quantities coincide. Other natural questions can be to determine the flooding time of the graph from a fixed vertex $x$ or its diameter, i.e. the maximum of the shortest paths between $x$ and all other vertices and the maximum of the flooding times, respectively. This paper investigates FPP on the inhomogeneous random graph (IHRG) model introduced in \cite{BB_IHRGM} with independent identically distributed (i.i.d.) exponential edge weights with rate one.

The addition of edge weights on the network can be interpreted as the cost of carrying the flow from one node to the other along the edge. Furthermore, edge weights can dramatically alter the geometry of the graph. For example, consider the complete graph on $n$ vertices first without edge weights. The hopcount between any two vertices is of course one. However, by adding i.i.d. exponential Exp(1) or uniform U(0,1) edge weights, the weight of the shortest-weight path is of order $\log n/n\ll1$ and the hopcount is about $\log n$ \cite{Janson_CompleteGraph}. A similar phenomena can be observed for the IHRG, see Subsection \ref{subsec::mainres}.

The proofs usually rely on results from branching processes. The use of exponential weights imply that the exploration processes of the graph are Markovian. Only recently was FPP studied on random graphs with general continuous edge weights \cite{RvdH_FPP_Generalweights}. Other related results are discussed in Subsection \ref{subsec::relatedwork}. We begin by introducing the IHRG model in Subsection \ref{subsec::model}. This section is continued with the statements of our main results in Subsection \ref{subsec::mainres} and then the main ideas of the proofs are sketched in Subsection \ref{subsec::sketch&structure}.

\subsection{The model}\label{subsec::model}

The random graph model we consider is a general inhomogeneous random graph model introduced by Bollob\'as, Janson and Riordan in \cite{BB_IHRGM}. We briefly describe the model $G(n, \kappa)$ on $n$ vertices and kernel $\kappa$ in the general setting and then turn to an important special case.

Each vertex of the graph will be assigned a type from a separable metric space $\CS$ which is equipped with a Borel probability measure $\mu$. For each $n$ we have a deterministic or random sample of $n$ points $\mathbf{x}_n=(x_1,\ldots,x_n)$ from $\CS$. We assume that the empirical distribution
\[\nu_n:= \dfrac{1}{n}\sum_{i=1}^n\delta_{x_i}\]
converges in probability to $\mu$ as $n\to\infty$, where $\delta_x$ is the measure consisting of a point mass of weight 1 at $x$. This convergence condition is equivalent to the condition
\begin{equation}\label{eq::n_t/n_to_mugeneral}
\nu_n(S):= \dfrac{\#\{i: x_i\in S\} }{n} \toinp \mu(S),
\end{equation}
for every $\mu$-continuity set $S\subset\CS$ (i.e. $S$ is measurable and there is no mass on the boundary of $S$). The pair $(\CS,\mu)$ is called a ground space and for a sequence $(\mathbf{x}_n)_{n\geq1}$ satisfying \eqref{eq::n_t/n_to_mugeneral} we say that the triplet $(\CS,\mu,(\mathbf{x}_n)_{n\geq1})$ defines a vertex space {\large$\nu$}. Further, a kernel $\kappa$ on a ground space is a symmetric non-negative measurable function on $\CS\times \CS$. The natural interpretation of $\kappa$ is that it measures the density of edges.

The simple random graph $G(n,\kappa)$ for a given kernel $\kappa$ and vertex space {\large$\nu$} is defined as follows. An edge $\{ij\}$ (with $i\neq j$) exists with probability
\begin{equation}\label{eq::edge_prob_pijgeneral}
  p_{ij} := \min\left\{\frac{\kappa(x_i, x_j)}{n}, 1\right\}.
\end{equation}
Independently of all randomness each edge is given an Exp(1) edge weight.

We want to exclude cases where the vertex set of $G(n,\kappa)$ can be split into two parts so that the probability of an edge from one part to the other is zero, i.e. we want to ensure the emergence of a single giant component later. To do so, further restrictions are needed for the kernel $\kappa$.
\begin{definition}\label{def::irred_kernel}
A kernel $\kappa$ on a ground space $(\CS,\mu)$ is irreducible if
\begin{equation*}
A\subseteq \CS \text{ and } \kappa=0 \text{ a.e. on } A\times(\CS\setminus A) \text{ implies } \mu(A)=0 \text{ or } \mu(\CS\setminus A)=0.
\end{equation*}
As a slight modification we say that $\kappa$ is quasi-irreducible if there is a $\mu$-continuity set $\CS'\subseteq\CS$ with $\mu(\CS')>0$ such that the restriction of $\kappa$ to $\CS'\times\CS'$ is irreducible and $\kappa(x,y)=0$ if $x\not\in\CS'$ or $y\not\in\CS'$.
\end{definition}
$G(n,\kappa)$ is a sparse graph, i.e. the number of edges $e(G(n,\kappa))$ is linear in $n$, since
\[\expect{e(G(n,\kappa))} = \expect{ \sum_{i<j}\min\left\{\kappa(x_i, x_j)/n, 1\right\} },\]
tends to $n\big(\frac{1}{2} \iint \kappa\big)$ under certain conditions ensuring this "well-behavior" (see \cite[Lemma 8.1]{BB_IHRGM}). This is formulated in the notion of graphical kernels.
\begin{definition}\label{def::graphical_kernel}
A kernel $\kappa$ is graphical on a vertex space $(\CS,\mu,(\mathbf{x}_n)_{n\geq1})$ if it is continuous  almost everywhere (a.e.) and in $L^1(\CS\times\CS,\mu\times \mu)$, furthermore
\begin{equation}\label{eq::expected_nof_edges}
      \dfrac{1}{n} \expect{e(G(n,\kappa))} \to \dfrac{1}{2} \iint_{\CS^2} \kappa(x,y)\mathrm{d}\mu(x)\mathrm{d}\mu(y).
      \end{equation}
\end{definition}
For example, condition \eqref{eq::expected_nof_edges} holds whenever $\kappa$ is bounded and $\nu$ is a vertex space. An important ingredient in the proof will be the use of approximating kernels, where $\kappa$ depends on $n$. We say that a sequence $\kappa_n$ of kernels on $(\CS,\mu)$ is graphical on $\nu$ with limit $\kappa$ if, for a.e. $(y,z)\in\CS^2$,
\[y_n\to y \text{ and } z_n\to z \text{ imply that } \kappa_n(y_n,z_n)\to \kappa(y,z),\]
$\kappa \in L_1$ and continuous a.e., and
\begin{equation}\label{eq::expected_nof_edges_general}
  \dfrac{1}{n} \expect{e(G(n,\kappa_n))} \to \dfrac{1}{2} \iint_{\CS^2} \kappa(x,y)\mathrm{d}\mu(x)\mathrm{d}\mu(y).
\end{equation}

The approximation of a general kernel will be done with an appropriate sequence of step functions. This motivates the special case of regular finitary kernels: the type-space $\CS$ has a finite partition into ($\mu$-continuity) sets $S_1, \ldots,S_r$ such that $\kappa$ is constant on each $S_i\times S_j$ for all $1\le i,j\le r$. By identifying each $S_i$ with a single point $i$ with weight $\mu_i=\mu(S_i)$, a random graph $G(n,\kappa)$ generated by a regular finitary kernel has the same distribution as a finite-type graph. If the type-space $\CS = \{1,2,\ldots,r\}$, and $n_t$ stands for the number of type $t$ vertices (so $\sum_{t\in \CS}n_t=n$), condition \ref{eq::n_t/n_to_mugeneral} becomes
\begin{equation}\label{eq::n_t/n_to_mu}
  \dfrac{n_t}{n} \toinp \mu_t \quad \text{holds for every } t\in\CS,
\end{equation}
and $\kappa= \left(\kappa(s,t)\right)_{s,t=1}^r$ is a symmetric non-negative $r \times r$ matrix. Note that finite-type kernels are automatically graphical. Further note that the Erdős-Rényi (ER) random graph is a special case of a finite-type graph when $r=1$ and $\kappa = c$. Then the probabilities $p_{ij}$ are all simply $c/n$ (for $n>c$).

We make an important assumption on the random graphs $G(n,\kappa)$. We will assume that asymptotically the average degree of a vertex is independent of its type. This is  referred to as the homogeneous case in \cite[Example 4.6]{BB_IHRGM}. In this case the global behavior of $G(n, \kappa)$ in the limit is the same as of the ER random graph, but the local behavior can be quite different. In the general setting this assumption can be formulated as
\begin{equation}\tag{AG}\label{ass:integral_la+1_general}
\int_{\CS} \kappa(x,y)\mathrm{d}\mu(y) = \wla+1 \quad \text{for a.e. } x.
\end{equation}
An example for a kernel $\kappa$ satisfying \eqref{ass:integral_la+1_general} can be given by taking $\CS=(0,1]$ (interpreted as the 1 dimensional torus $\mathbb{T}^1$), $\mu$ as the Lebesgue measure and $\kappa(x,y)=h(d(x,y))$ for an even function $h\geq0$ of period 1, where $d(x,y)$ is the metric given on $\CS$.
Figures \ref{fig::homogen1} and \ref{fig::homogen2} show the contours of two such examples, where the dark purple strips indicate where the contour is zero.

\begin{figure}[ht]
\begin{minipage}[ht]{0.47\columnwidth}
\centerline{\includegraphics[width=0.99\columnwidth]{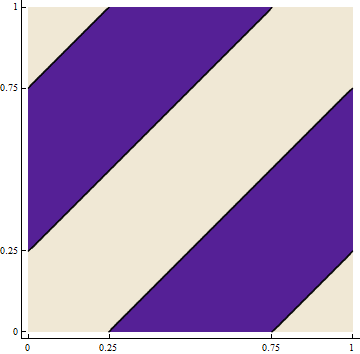}}
\caption{$h(z)=\indd{\,|z|<0.25}$  \label{fig::homogen1}}
\end{minipage}
\hfill
\begin{minipage}[ht]{0.47\columnwidth}
\centerline{\includegraphics[width=0.99\columnwidth]{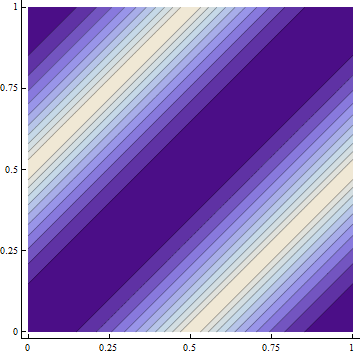}}
\caption{$h(z)=z^2$\label{fig::homogen2}}
\end{minipage}
\end{figure}

In the finite-type case we use the following notation:
\begin{equation}\label{eq::lambda_st_def}
  \lambda_{st} := \kappa(s,t)\mu_t,
\end{equation}
for all $s,t\in\CS$. The number of type $t$ neighbors of a type $s$ vertex are binomially distributed with parameters $n_t-\delta_{st}$ and $\kappa(s,t)/n$ ($\delta_{st}=1$ if and only if $s=t$). Thus from \eqref{eq::n_t/n_to_mu} we get that in the limit $\lambda_{st}$ gives us the average number of type $t$ neighbors of a type $s$ vertex. From here we construct the matrix
\be\label{eq::A_matrix}
  A = 
  \begin{pmatrix}
    \lambda_{11}-1 & \lambda_{12}   & \cdots & \lambda_{1r} \\
    \lambda_{21}   & \lambda_{22}-1 & \cdots & \lambda_{2r} \\
    \vdots  & \vdots  & \ddots & \vdots  \\
    \lambda_{r1}   & \lambda_{r2}   & \cdots & \lambda_{rr}-1
  \end{pmatrix}.
\ee

\begin{assumption}\label{ass::sorosszeg_la}
We will assume that the row sums of the matrix $A$ are the same and equal to $\wla>0$. Also assume that $A$ is irreducible, i.e. there exists a $k_0$ s.t. $A^{k_0}$ has strictly positive entries.
\end{assumption}
In the finite-type case Assumption \ref{ass::sorosszeg_la} is the equivalent of \eqref{ass:integral_la+1_general}. The $\wla>0$ condition is necessary and sufficient for a giant component to emerge in $G(n,\kappa)$ \cite[Theorem 3.1]{BB_IHRGM}. Let $\pi$ denote the normalized left eigenvector corresponding to the eigenvalue $\wla$ of $A$. Under Assumption \ref{ass::sorosszeg_la} we find that
\begin{equation}\label{eq::pi==mu}
\pi=\mu.
\end{equation}
Indeed, using the symmetry of $\kappa$,  $\mu A=\wla\mu$ follows immediately.

\vspace{1mm}
Let us introduce some standard notation. Let $\text{Bin}(n,p), \,\text{Poi}(\lambda), \,\text{Exp}(\mu)$ respectively denote a binomial, a Poisson and an exponential random variable with the parameters having their usual meaning. Convergence almost surely, in distribution and in probability are denoted by $\toas, \,\toindis, \,\toinp$ respectively. A sequence of events holds with high probability (whp), if it holds with probability tending to $1$ as $n\to\infty$. We use the Landau symbols $O$ and $o$ with their usual meaning. We say that a sequence of random variables $X_n$ satisfies $X_n=O_{\PP}(b_n)$ if $X_n/b_n$ is tight (i.e. $\forall \varepsilon>0 \, \exists M: \, \prob{X_n>M b_n}<\varepsilon$) or $X_n=o_{\PP}(b_n)$ if $X_n/b_n\toinp 0$. Now we turn to the main results.
\vspace{1mm}

\subsection{Main results}\label{subsec::mainres}

We investigate the weight and the number of edges on the shortest-weight path between two uniformly picked vertices $x$ and $y$. Let $\Gamma_{xy}$ denote the set of all $\pi$ paths in $G(n,\kappa)$ between $x$ and $y$. Denote the weight of the shortest-weight path by
\begin{equation}\label{eq::P_ndef}
\CP_n = \min_{\pi\in\Gamma_{xy}} \sum_{e\in\pi}X_e,
\end{equation}
where $X_e$ is the exponential edge weight attached to edge $e$ in the construction of $G(n,\kappa)$. Let $\CH_n$ denote the number of edges or hopcount of this path. If the two vertices are in different components of the graph, then let $\CP_n, \CH_n=\infty$. The following theorem describes the asymptotic behavior of these two quantities.

\begin{theorem}[Asymptotics of Hopcount \& Shortest weight] \label{thm::GeneralSetting}
Let $(\CS,\mu)$ be an arbitrary ground space and $\kappa$ be a uniformly continuous, quasi-irreducible, graphical kernel on $(\CS,\mu)$ that satisfies $\sup \kappa(x,y)<\infty$ and
\[\int_{\CS} \kappa(x,y)\mathrm{d}\mu(y) = \wla+1<\infty \quad \text{for a.e. } x\in\CS.\]
Then the hopcount $\CH_n$  and the minimal weight $\CP_n$ between two uniformly chosen vertices, conditioned on being connected, satisfy
\begin{equation*}
  \left(\frac{\CH_n- \frac{\wla+1}{\wla} \log n}{\sqrt{\frac{\wla+1}{\wla}\log n}} ,\; \CP_n - \frac{1}{\wla}\log n\right)\toindis (Z,L),
\end{equation*}
where $Z$ is a standard normal variable. Furthermore, $L$ is a non-degenerate real valued random variable whose distribution can be precisely determined from the behavior of the multi-type branching process that arises when exploring a component of $G(n, \kappa)$.
\end{theorem}

A possible extension is to have $\kappa$ depend on $n$. In this case we have a sequence of matrices $A_n$, each satisfying Assumption \ref{ass::sorosszeg_la} with $\wla_n$ as the sum of the rows. Then we can trivially extend our theorems in the following form:

\begin{corollary}\label{cor::MainWithLambda_n}
If $\lim_{n\to \infty} \wla_n = \wla<\infty$ then for the hopcount we have
\[\left(\frac{\CH_n- \frac{\wla+1}{\wla} \log n}{\sqrt{\frac{\wla+1}{\wla}\log n}} ,\; \CP_n - \frac{1}{\wla}\log n \right) \toindis (Z, L),\]
where $Z$ and $L$ are as in Theorem \ref{thm::GeneralSetting}.
\end{corollary}

It is interesting to compare these results with its counterpart where there are no edge weights. Then of course $\CH_n=\CP_n$. Theorem 3.14 of \cite{BB_IHRGM} in our context says that in this case the typical distance between two vertices in the giant component still scales as $\log n$, but with a different constant: $1/\log(\wla +1)$. Comparing the constants, we see that
\[\frac{1}{\wla} <\frac{1}{\log (1+\wla)}< \frac{\wla +1}{\wla}.\]
This shows that by adding edge weights the structure of the graph changes. Along the shortest-weight path more vertices are visited ($\CH_n$) than on the path with the least number of vertices. At the same time the weight of the path ($\CP_n$) becomes smaller than the number of edges on the path with the least number of vertices.
The figure below illustrates this, where the red path is the shortest-weight path while the green is the one with the least number of vertices.

\begin{center}
  \begin{tikzpicture}
  [scale=0.45, yscale=0.9,
   csucs/.style={fill=red!20, circle},
   egy/.style={circle, shade, ball color=blue},
   ket/.style={circle, shade, ball color=red},
   har/.style={circle, shade, ball color=green},
   rotate=90, thick, sibling distance=2cm]

  \node[csucs] (x)  at (0,0)   {$x$}
    child[blue] {node[egy] {} child[grow=-115, level distance=2.5cm]{node[egy] {} child[grow=-90]{node[egy] {} child{node[egy] {} child[grow=-80]{node[egy] {} child[grow=-70]{node[egy] {} child[level distance=1.5cm, grow=-60]{node[egy] (A) {} }}}}}}}
    child[blue, level distance=2.5cm] {node[egy] {} child[grow=-105]{node[egy] {} child[grow=-90, level distance=3.5cm]{node[egy] {} child[level distance=1.5cm]{node[egy] {} child[level distance=2.5cm, grow=-75]{node[egy] (B) {} }}}}}
    child[red, level distance=2cm] {node[ket] {} child{node[ket] {} child{node[ket] {} child[level distance=3cm]{node[ket] {} child[level distance=1.5cm]{node[ket] {} child[level distance=2.5cm]{node[ket] (C) {} }}}}}}
    child[blue, level distance=2cm] {node[egy] {} child[grow=-75]{node[egy] {} child[grow=-90, level distance=1.5cm]{node[egy] {} child{node[egy] {} child[level distance=2.5cm]{node[egy] {} child[grow=-105]{node[egy] (D) {} }}}}}}
    child[green] {node[har] {} child[grow=-70, level distance=3cm]{node[har] {} child[grow=-90, level distance=4.5cm]{node[har] {} child[grow=-110, level distance=4cm]{node[har] (E) {} }}}};

  \node[csucs] (y)  at (0,-16) {$y$} ;
  \foreach \x in {A, B, D}
    \draw[blue] (y) -- (\x);
  \draw[red] (y) -- (C); \draw[green] (y) -- (E);

  \draw[green] (2.4,-17.4) -- (2.4,-18.2) node[har] {} -- (2.4,-19) node[right, black, text width=3.8cm] {\phantom{x} shortest path \textbf{without} edge weights};
  \draw[red] (0,-17.4) -- (0,-18.2) node[ket] {} -- (0,-19) node[right, black, text width=3.8cm] {\phantom{x} shortest path \textbf{with} edge weights};
  \draw[blue] (-2.4,-17.4) -- (-2.4,-18.2) node[egy] {} -- (-2.4,-19) node[right, black] {other paths};
\end{tikzpicture}
\end{center}

Similarly as in \cite{RvdH_FPP} we also investigate the dense regime, where $\lim \wla_n = \infty$, i.e. the average degree tends to infinity. In this case whp any two vertices are connected, so the giant component contains $n(1-o(1))$ vertices (see \cite{BB_IHRGM}). Again comparing with the counterpart without edge weights, the change in the graph structure is even more significant. Without edge weights the graph is ultra small, meaning that graph distances between uniformly chosen vertices are $o(\log n)$. However with the addition of edge weights the following theorem states that even the magnitudes do not coincide. We get the same type of behavior as observed in the sparse setting, which means that on the shortest-weight path many more vertices are visited.

\begin{theorem}[Dense setting]\label{thm::DenseSettingMain}
Under Assumption \ref{ass::sorosszeg_la} and $\lim\limits_{n\to\infty}\wla_n = \infty$, then we have
\begin{equation*}
    \left(\frac{H_n- \frac{\wla_n+1}{\wla_n} \log n}{\sqrt{\log n}}, \wla_n \mathcal P_n - \log n\right)\toindis (Z,\widetilde L),
  \end{equation*}
  where $Z$ is a standard normal variable, $\widetilde L$ is equal in distribution to the sum of independent random variables $Y_1+Y_2-Y_3$, with $Y_i$ i.i.d. standard Gumbel random variables. Further, we can substitute $\frac{\wla_n+1}{\wla_n}$ in the centraling of the hopcount by $1$ if and only if $\wla_n/\sqrt{\log n} \to \infty$.
\end{theorem}

\subsection{Sketch of proofs}\label{subsec::sketch&structure}

When searching for the shortest-weight path between two vertices, the intuitive picture to keep in mind is to let fluid percolate at a constant rate simultaneously from each vertex. After time $\tau$ the flow from $x$ contains the vertices whose shortest-weight path is at most $\tau$ from $x$. It is standard to relate this exploration process of the neighborhood of a vertex to a branching process (BP). The generation and lifetime of a particle in a BP corresponds to its hopcount and weight of optimal path in the exploration process. Section \ref{sec::BP} introduces a continuous-time multi-type branching process that arises naturally in an exploration process of the giant component of $G(n, \kappa)$. An analog of the main results is proved in Theorem \ref{thm::MainResForBP} in the BP setting.

For the results to carry through we need an embedding of the BP into the IHRG. This is dealt with in Section \ref{sec::embedding}. The vertices of the BP and the vertices of the IHRG need to be identified with one another with high probability. To achieve this labels are assigned to the vertices according to their type in such a way that we can rule out loop and multiple edges in the IHRG. To deal with the problem of cycles in the graph, the notion of thinning is introduced on branching processes.

Intuitively it is clear that when the two flows of fluid collide, then the shortest-weight path between $x$ and $y$ has been found. Thus it is crucial to determine when the connection actually happens. The random time of the collision will be referred to as the connection time $C_{xy}$. In the figure below, the fluids flow continuously and the vertices $v_x$ and $v_y$ denote the last wetted vertices (along the shortest path) by the flows from $x$ and $y$ respectively before time $C_{xy}$.
\begin{center}
\begin{tikzpicture}
  [rotate=90, scale=0.45, xscale=0.75, yscale=1.1, csucs/.style={fill=red!20, circle}, level distance= 22mm, sibling distance=20mm]

  \node[csucs] (x)  at (0,0)   {$x$} child child[missing] child;
  \node[csucs] (y)  at (0,-14) {$y$} [grow'=up] child child child[missing] child;
  \node[csucs] (vx) at (-0.5,-4.4) {$v_x$};
  \node[csucs] (vy) at (0.5,-8.5)  {$v_y$};

  \draw[very thick]
      (x) -- (vx)  (vy) -- (y);
  \draw[very thick] (vx) -- (vy);

  \foreach \col/\x in {black!20/-2, black!50/-4, black/-7}
    \draw[\col,thick] (\x,0) arc (-180:0:-\x);
  \foreach \col/\x in {black!20/-2, black!50/-4, black/-7}
    \draw[\col,thick] (-\x,-14) arc (0:180:-\x);

  \foreach \x/\old/\z/\zz in {-8/below/A/Exp(1),8/above/G/$+1$}
    {\draw (\x,0)  -- (\x,-14);
     \foreach \y in {0,-4.4,-8.5,-14}
       \filldraw (\x-0.3,\y-0.05) rectangle (\x+0.3,\y+0.05);
     \foreach \y/\szoveg in {-2.2/$\z(v_x)$,-6.5/\zz,-11.2/$\z(v_y)$}
       \node[\old] at (\x,\y) {\szoveg}; };
  \node[right, text width=2cm] at (-8,-14.2) {weight of path};
  \node[right, text width=2cm] at (8,-14.2) {number of edges};
  \node[right, text width=3.5cm] at (0,-14.6) {{\large Colliding flows of fluid}};
\end{tikzpicture}
\end{center}
Let $G(v_x)$ denote the number of edges between the source $x$ and vertex $v_x$ and $A(v_x)$ the weight of this path. Then for the hopcount $\CH_n(x,y)$ and for the weight $\CP_n(x,y)$ of the shortest-weight path we get that
\begin{align*}
  \CH_n(x,y) &= G(v_x)+G(v_y)+1, \\
  \CP_n(x,y) &= A(v_x)+A(v_y)+\mathrm{Exp}(1).
\end{align*}
The problem is that we can't let the fluids flow continuously, but only in discrete steps as new vertices are wetted by the flows. Thus, the sum of the path length plus the remaining edge-weight between $v_x$ and $v_y$ must be minimized over all possible choices of last vertices $v_x$ and $v_y$. The rigorous treatment of the connection time is done in Section \ref{sec::ConnectionTime}

After these preparations we begin Section \ref{sec::proof} with the proof of Theorem \ref{thm::DenseSettingMain} for finite-type graphs. To prove the main result in the general setting a sequence of finite partitions of $\CS$ is used. The approximating step function on $\CS\times\CS$ will be given by the average of the kernel $\kappa$ determined by the partition. This ensures that if $\kappa$ is homogeneous then so are all of its approximations.

\subsection{Related work}\label{subsec::relatedwork}

First passage percolation has been investigated on various models, such as the integer lattice, the mean-field model, configuration model, graphs with i.i.d. vertex degrees both without edge weights (see e.g. \cite{RvdH_Bhamidi, RvdH_Hoog_Znamen, Fernholz_Ramachandran, RvdH_Hoog_Mieg, Chung_Lu1, BB_IHRGM, BB_delaVega, Britton_Deijfen_Lof, Norros_Reittu}) and with exponential weights (see e.g. \cite{Amini_Lelarge, Janson_CompleteGraph, RvdH_FPPonCM, Bhamidi, RvdH_FPP, Howard}). The list of results are far from complete, we only attempt to discuss the ones directly related to the present paper.

The IHRG model was extensively investigated by Bollobás, Janson and Riordan in \cite{BB_IHRGM} from many different aspects, including typical distances without edge weights. We already showed what effect the addition of edge weights has on the structure of the graph in Subsection \ref{subsec::mainres}. Many other models are closely related to the IHRG model, for full details see \cite[Sections 4 and 16]{BB_IHRGM}.

The classical supercritical Erdős-Rényi (ER) random graph $G(n,c/n)$ ($c>1$) is the special case when $|\CS|=1$. Our results generalize the FPP results of Bhamidi, van der Hofstad and Hooghiemstra on ER random graphs with i.i.d. exponential edge weights in \cite{RvdH_FPP}. Finite-type graphs were previously studied by Soderberg \cite{Soderberg}.

A restrictive, yet natural class of inhomogeneous random graphs is the \textit{rank-1} class. The kernel $\kappa$ has the special form $\kappa(x,y)=\phi(x)\phi(y)$, where the positive function $\phi$ on $\CS$ can be interpreted as the "activity" of a type-$x$ vertex. In the Chung-Lu model each vertex $i$ is given a positive weight $w_i$ and the edge probabilities $p_{ij}$ are given by $p_{ij}:=w_iw_j/\ell_n$, where $\ell_n=\sum_{i=1}^nw_i$. Norros and Reittu \cite{Norros_Reittu} give results on the existence and size of a giant component with random $w_i$. For deterministic $w_i$, Chung and Lu \cite{Chung_Lu1, Chung_Lu2} show that, under certain conditions, (without edge weights) the typical distance between two vertices is $\log n/\log \bar{d}$, where $\bar{d}=\sum w_i^2/\ell_n$.

A closely related model is the generalized random graph introduced by Britton, Deijfen and Martin-Löf \cite{Britton_Deijfen_Lof} with edge probabilities $p_{ij}=w_iw_j/(n+w_iw_j)$. They show that conditioned on the vertex degrees, the resulting graph is uniformly distributed over all graphs with the given degree sequence. As a result Bhamidi, van der Hofstad and Hooghiemstra in \cite{RvdH_FPP_Generalweights} prove FPP results for the latter two models with \textit{general} continuous edge weights. This is a corollary of their result for the configuration model.


\section{Multi-type branching processes}\label{sec::BP}

In this section we collect the needed properties of the branching process that arises when exploring a component of $G(n,\kappa)$. Let us define  a multi-type continuous time branching process with type space $\CS$, where a particle of type $x\in\CS$, when it splits, gives birth to a set of offsprings distributed as a Poisson process on $\CS$ with intensity measure $\kappa(x,y)\mathrm{d}\mu(y)$. That is, the number of children with types in a subset $S\subset \CS$ has a Poisson distribution with mean $\int_S \kappa(x,y)\mathrm{d}\mu(y)$. Each offspring lives for an Exp(1) amount of time independently of everything else. We denote this branching process with root of type $s$ up to time $t$ by $\Psi_\kappa^s(t)$. For a set $E\in \CS$, $\CD_m^E$ and $\CA_m^E$ stands for the set of dead and alive particles after the $m$-th split whose type belongs to the set $E\in \CS$. For $E = \CS$ we simply write $\CD_m$ and $\CA_m$, respectively. Similarly, $S_m^E$ and $N_m^E$ stand for the number of alive and dead individuals with type in the set $E\in S$, and we simply write $S_m$ for $E = \CS$. Let us also write $\tau_m$ for the time of the $m$-th split.

The branching process $\Psi_\kappa^x$ arises naturally when exploring a component of $G(n,\kappa)$, analogously to that of the exploration of FPP on the Erd\H os-R\'enyi graph. In this exploration process no size-biasing of the degrees happens, due to the independence of edges.

\vspace{1mm}

Analogous to the lines of \cite{BB_IHRGM}, we define an integral operator, whose norm establishes a direct connection with the emergence of a giant component in the random graph $G(n,\kappa)$ and the survival of $\Psi_\kappa$. Let
\begin{equation}\label{def::Tkappa_operator}
(T_{\kappa}f)(x) := \int_{\CS} \kappa(x,y)f(y)\mathrm{d}\mu(y),
\end{equation}
for any measurable function $f$ such that this integral is defined.
The norm of $T_{\kappa}$ is
\[\|T_{\kappa}\| := \sup \left\{ \|T_{\kappa}f\|_2: f\geq0,\; \|f\|_2\leq1 \right\}\leq\infty, \]
where $\|\cdot\|_2$ is the norm of $L^2(\CS,\mu)$.

In the finite-type case each type-$t$ particle gives birth to a Poisson number of type $s$ children with parameter $\la_{st}= \kappa(s,t)\mu(t)$, so in this case $T_{\kappa}f = (A+I)f$ with $A$ defined in \eqref{eq::A_matrix}. Easy calculations yields the norm:
\begin{align*}
\|T_{\kappa}\| = \|A+I\| = \wla+1,
\end{align*}
where the last equality holds because of Assumption \ref{ass::sorosszeg_la}.

Let us recall Theorem 3.1 of \cite{BB_IHRGM}: a giant component emerges in $G(n,\kappa_n)$ with $\kappa_n \to \kappa$ and its corresponding branching process survives with positive probability if and only if $\|T_{\kappa}\|>1$. This is why we assumed throughout that $\wla>0$. More precisely, the survival probability of the branching process is the maximal solution to the functional equation $\rho(x)= 1- \exp\{-T_\kappa\rho(x)\}$. Under assumption \eqref{ass:integral_la+1_general}, the maximal solution is independent of the type $\rho(x)=\rho= 1- \exp\{-(\wla+1) \rho\}$.

Throughout the proofs, we will mainly work with finite-type case $\CS=\{ 1,\dots,r\}$ given with the mean-offspring matrix $A$ in \eqref{eq::A_matrix}. Recall that $\pi$ stands for the normalized left main eigenvector of $A$ (i.e. $\pi$ is a probability measure on $\CS$).  We will need the following limit theorems for $S_m^s$ and $N_m^s$, the number of alive and dead individuals of a given type $s\in\CS$: (these results do not require Assumption \eqref{ass::sorosszeg_la}), and all can be found in Athreya-Ney \cite{Athreya_BP}.

\begin{theorem}[BP-asymptotics]\label{thm::asyptoticsForBP}
For a finite-type continuous time Branching Process as defined above,
\begin{enumerate}[(i)]
\item On the set of non-extinction as $m\to\infty$
    \begin{equation}\label{eq::S_m^t/m}
      \lim_{m\to\infty}\frac{\big(S_m^{1}, \ldots, S_m^{r}\big)}{\wla m}  = \pi \text{ a.s. ,}
    \end{equation}
    \item and similarly
    \begin{equation}\label{eq::N_m^t/m}
       \frac{\big(N_m^{1}, \ldots, N_m^{r}\big)}{m} \toas \pi.
    \end{equation}
\end{enumerate}
\end{theorem}
As a consequence we also get that
\begin{equation}\label{eq::S_m/m}
\dfrac{1}{m}S_m \to \wla, \quad \mbox{and}\quad  \dfrac{S_m^{t}}{S_m} = \, \stackrel{\eqref{eq::S_m^t/m}}{\longrightarrow} \, \pi_s \ a.s. \\
\end{equation}

Estimates can also be given for the magnitude of the error terms in \eqref{eq::S_m/m}. More precisely, the exponent of the error terms is of order $0.5 \wedge \left(\mathcal{R}e( \la_2) / \wla\right)$, where $\la_2$ is the eigenvalue of $A$ with second largest real part. If this quotient is less then $1/2$, then $(S_m-\wla m)/\sqrt{m}$ and $(S^{t}_m-\wla\pi_t m)/\sqrt{m}$ tends to a multidimensional normal random variable. For details we refer to \cite[Theorems 3.22-24.]{Janson04functionallimit}. We will later make use of these in some of our arguments.

The generation of a particle in the Branching Process corresponds to the hopcount of the vertex in the IHRG. Thus, when we establish the connection between the exploration processes of $x$ and $y$, the hopcount of the vertices at the connection are needed. We will later see that the generation of the particles in the two BP-s will be only independent when conditioned on their type. The following lemma for the multi-type BP will handle these issues:

\begin{theorem}[Generation of a uniformly picked particle in a given type-set]\label{thm::MainResForBP}
\
Let $G^E_m$ denote the generation of a uniformly picked individual from $\CA^E(m)$, $E\subset \CS$. Then, conditioned on survival, the following holds for $m\to\infty$:
\begin{equation}\label{eq::G_mConvForBP}
  \frac{G_m^E- \frac{\wla+1}{\wla} \log m}{\sqrt{\frac{\wla+1}{\wla}\log m}}\toindis Z,
\end{equation}
where $Z$ is a standard normal variable.
\end{theorem}
Thus, the generation of a uniformly picked alive of a given type follows a central limit theorem with parameters independent of the type.
The lemma is a consequence of the result by Kharlamov \cite{Kharlamov1}. He proved the result in continuous time (i.e. for $G_t^E$), for an arbitrary set of types $E\in \CS$ under the conditions that the type-distribution is tending to the stationary $\pi$ exponentially fast in time and that the expectation and variance of the generation of the density of type $y$ individuals after unit time for any type $y\in \CS$ is uniformly bounded in $y$. These conditions clearly hold if the life-time is exponential and the total number of children is Poi($\wla+1$).

To get the result from continuous time to discrete time, one needs to replace $t$ by the $m$-th split time $\tau_m$ and use aperiodicity of the types along generations.

For the asymptotic behavior of split times $\tau_m$ we cite again \cite{Athreya_BP}:
\begin{theorem}\label{thm::tau_m_limit}
\begin{equation}\label{eq::A_mConvForBP}
  \tau_m - \frac{1}{\wla}\log m \toas -\frac{1}{\wla} \log\left(\frac{1}{\wla}  W\right)
\end{equation}
on the set of non-extinction, and $\Pv_x(W>0) = \rho(x)$.
The distribution of $W$ is given by distribution of the limit of the continuous time martingale $\lim_{t\to \infty} e^{-\wla t} |\un S_t|= W$, where $\underline S_t$ is vector of alive particles in $\Psi_\kappa(t)$.
\end{theorem}
Under assumption \eqref{ass::sorosszeg_la}, it is not hard to see that the distribution of $W$ is independent of the initial type $x$, and is the same as for a single-type BP with Poi$(\wla)$ offsprings. Further, $\Ev(W)=1$ and the moment generating function $M_W(t)$ of  $W$ satisfies the functional equation
\be\label{eq::W_distribution} M_W(t)=\int_0^\infty \exp\left\{(\wla+1)\left(M_W(t e^{-\wla y })-1\right)\right\}e^{-y}\mathrm{d} y.\ee

In some calculations, we will need the following lemma for the total number of alive individuals. This lemma is slightly stronger than what follows from the almost sure convergence, and in fact this is the crucial element which is missing for the general $A$ matrix case, i.e. without the Assumption \eqref{ass::sorosszeg_la}.

\begin{lemma}\label{lemma::summable}[Summable error terms for $S_j$ ] Let us suppose Assumption \ref{ass::sorosszeg_la} holds. Then,
Conditioned on the survival, there exists a constant $C> 0$ such that
\be\label{eq::S_j_summable} |S_j - \wla j|\le C (j\log j)^{1/2} \quad \forall j\ge\log\log m \ee holds with high probability.
\end{lemma}
The lemma is the direct analog of \cite[Proposition 3.7]{RvdH_FPP}. The heuristics of the proof is the following: The total number of alives at the split times is just a random walk on $\Zb$ with independent Poi$(\wla+1)-1$ increments, but conditioning on survival destroys this independence. However, with a coupling argument we can reduce the error probabilities to the independent case as follows. First condition on the first $\log \log m$ steps of the walk to stay positive. On this event, the random walk is whp close to its expected value, and hitting zero from this point has exponentially small probability. Then, after the $\log\log m$-th step we can forget about the conditioning on survival and work with independent Poi$(\wla+1)-1$ increments. For these, a simple large deviation estimate already yields a summable error term in $j$ for the events in \eqref{eq::S_j_summable}.


\section{Embedding the BP into the IHRG}\label{sec::embedding}

This section relates the exploration process of the neighborhood of a vertex in the inhomogeneous random graph (IHRG) model $G(n,\kappa)$ of Section \ref{sec::intro&results} and the branching process $\Psi_\kappa$ of Section \ref{sec::BP} to one another. To obtain an embedding, we give each BP-particle a label (= the vertex it corresponds) to get a continuous-time labeled BP (CTLBP) on which we define thinning.  This procedure basically deals with the problem of finding the shortest weight path amongst multiple possible paths between any two given vertices.

\subsection{Labeled branching processes}\label{subsec::LBP}
The first step is to determine for a particle in $\Psi_\kappa$ to which vertex of the graph $G(n, \kappa)$ it corresponds. Thus, we describe continuous-time labeled branching processes (CTLBP). By assigning labels to the individuals in $\Psi_\kappa$ according to their type, we will be able to couple the BP to the exploration of $G(n, \kappa)$ from a given initial vertex. Fix $n\geq1$ and denote the set of labels (vertices of the graph) by $[n]=\{1,2,\ldots,n\}$. It is important to distinguish the labels according to the types, so $[n]$ is the disjoint union of the sets of labels $[n]^{(1)}, \ldots, [n]^{(r)}$, where there are $n_t$ different labels in $[n]^{(t)}$. An individual of type $s$ in $\Psi_\kappa$ will be assigned a label from $[n]^{(s)}$.

The construction goes as follows. Assume that the root of the BP is of type $s$. Assign the label $i_0\in [n]^{(s)}$ to it. The root immediately dies and gives birth to a random number $\eta_{st}$ of type $t$ offspring, for each $t \in \CS$.
By \eqref{eq::edge_prob_pijgeneral}, we get that the distribution of the number of type $t$ neighbors of a type $s$ vertex in $G(n,\kappa)$ is
\begin{equation*}
\eta_{st} \stackrel{d}{=} \text{Bin}\left(n_t-\delta_{st}, \dfrac{\kappa(s,t)}{n}\right).
\end{equation*} For $t\neq s$, the $\eta_{st}$ "new" individuals of type $t$ are assigned different labels from $[n]^{(t)}$ drawn without replacement uniformly at random. For $t=s$, we choose the labels the same way from $[n]^{(s)}\backslash i_0$.
 Further, by denoting the number of children of the $j$-th dying particle in the BP by $D_j^{Bin}$, we see that conditioned on the event that the $j$-th split is of a type $s$ vertex, the distribution of $D_j^{Bin}$ is a sum of independent random variables
\begin{equation}\label{eq::offspring_dist}
D_j^{Bin,(s)} = \eta_{s1}+ \eta_{s2}+\ldots+\eta_{sr}.
\end{equation}
Since the edge weights are exponential, each offspring lives for an exponentially distributed random time with rate one, then dies and gives birth to its children, thus the natural embedding requires a continuous time multi-type BP with offspring distribution given in \eqref{eq::offspring_dist}. We denote this BP up to time $t$ by $\Psi_\kappa^{Bin}(t)$.

For the number of alive vertices after the $m$-th split we write $S^{Bin}_m=D^{Bin}_1+\ldots+D^{Bin}_m-(m-1)$. Due to the memoryless property of the exponential distribution, the next vertex ($m+1$-st) to split is uniformly distributed among the $S^{Bin}_m$ alive vertices.

Inductively, the individual that splits at the $j$-th split time can be uniformly chosen from the $S^{Bin}_j$ alive vertices. Assume it is of type $\tilde s$ with label $i_j$. It gives birth to $\eta_{\tilde st}$ offspring of type $t$ and choose $\eta_{\tilde st}$ different labels from $[n]^{(t)}$ for $t\neq \tilde s$ and $\eta_{\tilde s\tilde s}$ different labels from $[n]^{(\tilde s)}\backslash i_j$ for $t= \tilde s$.

\begin{lemma}[Coupling error to Poi offsprings]\label{lemma::BINPOIcoupling}
 The multi-type branching processes $\Psi_\kappa^{Bin}$ and $\Psi_\kappa$ (described at the beginning of Section \ref{sec::BP} ) can be coupled until the $m$-th split with error
 \[ \Pv\left[\exists j \le m, \ D_j^{Bin} \neq D_j^{Poi} \right] \le \frac{m}{n} (\wla+1)\max \kappa (1+ o(1)).\]
As a consequence we immediately get that the time of the $m$-th split and the generation of a uniformly picked alive particle in a given type-set $E\in S$ can be coupled in the two processes with the same upper bound for the error probability.
\end{lemma}
\begin{proof}
By the usual coupling of Binomial and Poisson distribution, we can couple $\eta_{st}$ in \eqref{eq::offspring_dist} to $\xi_{st}\sim$Poi$(\la_{st})$ distribution with a coupling error of $n_t (\kappa(s,t)^2/n^2)= \la_{s_t} \kappa(s,t)(1+ o(1)) / n$. Under the assumption that $\max \kappa< \infty$, the coupling to Poisson offspring distribution of a single split has an error at most $\max \kappa (\wla+1) /n$. Summing up the error terms up to $m$ yields the statement.
Clearly the error probabilities $\Pv[G_m^{Bin,E}\neq G_m^{Poi,E}] $ and $\Pv[\tau_m^{Bin}\neq\tau_m^{Poi}]$ can be bounded from above by the error probability that the coupling fails between the two BP-s up to the $m$-th step.
\end{proof}

As a consequence of the lemma, we apply the above described labeling procedure to $\Psi_\kappa$, to embed it in $G(n,\kappa)$, and the error term stays small until we do $o(n/(\wla+1))$ steps. However, there is another error arising from the embedding: By assigning the labels the above described way, we rule out the possibility of loop edges and multiple edges between vertices during an exploration of $G(n, \kappa)$. The error of the coupling may arise from cycles: we have to find  the shortest weight path amongst multiple possible paths between any two given vertices. We handle this problem by thinning $\Psi_\kappa$.

\begin{center}
\begin{tikzpicture}
[scale=0.4,
 fiu/.style={circle, shade, ball color=blue!30, scale=0.9},
 lany/.style={circle, shade, ball color=red!30, scale=0.9},
 el/.style ={thick, double = black, double distance = 1pt},
 level 1/.style={sibling distance=5cm},
 level 2/.style={sibling distance=2cm}, thick]

\begin{scope}[yscale=0.7]
\node[fiu] {\textbf{1}}
  child[level distance=1.5cm] {node[fiu]  {\textbf{2}}
    child[level distance=7.5cm] {node[lany] {\textbf{6}} } child[level distance=3cm] {node[lany] {\textbf{7}} child[level distance=12cm] {node[lany] {\textbf{5}} child[level distance=4.5cm] {node[lany] {\textbf{7}} } } child[level distance=6cm] {node[fiu] {\textbf{4}} } } }
  child[level distance=6cm] {node[lany] {\textbf{5}}
    child[level distance=9cm] {node[lany] {\textbf{7}} } child[level distance=12cm] {node[fiu] {\textbf{3}} } child[level distance=13.5cm] {node[fiu] {\textbf{1}} } }
  child[level distance=3cm] {node[fiu]  {\textbf{3}}
    child[level distance=4.5cm] {node[fiu] {\textbf{4}} child[level distance=4.5cm] {node[lany] {\textbf{6}} } } child[level distance=10.5cm] {node[lany] {\textbf{5}} } };
\node[yshift=-6.5cm] at (0,0) {{\large Marked BP}};
\end{scope}

\begin{scope}[xshift=-15cm,yshift=-5cm]
\foreach \x/\y in {0/1, 50/2, 100/3, 150/4}
  \node[fiu] (\y) at (\x:5cm) {\textbf{\y}};
\foreach \x/\y in {200/5, 250/6, 300/7}
  \node[lany] (\y) at (\x:5cm) {\textbf{\y}};

\draw[el, blue]  (4) -- (3) -- (1) -- (2);
\draw[el, red]   (7) -- (5) -- (6);
\draw[el, purple](1) -- (5) -- (3);
\draw[el, purple](7) -- (4) -- (6) -- (2) -- (7);
\node[yshift=-2.8cm] at (0,0) {{\large$G(n,\kappa)$}};
\node[yshift=-3.2cm, text width=5cm, below] at (0,0) 
{\begin{itemize}
 \item Marks of siblings are different
 \item Parent-offspring marks are different
 \end{itemize}};
\end{scope}
\end{tikzpicture}
\end{center}

\subsection{Thinning the branching process}\label{subsec::thinning}

The notion of thinning is introduced on branching processes to identify the shortest weight path between any two given vertices. The percolating fluid first reaches $y$ along the shortest path, then other paths are found later whenever the label of $y$ reappears in the labeling procedure. Thus, we only have to keep track of the first occurrences of each label.

In terms of the CTLBP this means that when we reach a label
\[ i_k\in\CD(k-1):=\{ i_0, i_1, \ldots, i_{k-1}\},\]
then we found a cycle in the exploration of the IHRG. The longer paths are irrelevant, thus we delete $i_k$ and the whole subtree starting from it in $\Psi_\kappa$. We call the label $i_k$ and its subtree thinned. As a result we only keep the shortest weight paths between pairs of vertices. We refer to the resulting  process as $\Th\Psi_\kappa$.

For a fixed $n$, the total number of labels is finite ($=n$), so a.s. at some random time all labels will have appeared. This means that $\Th\Psi_\kappa$ dies out a.s., and at this time we found the minimal weight spanning tree of a component of $G(n,\kappa)$. It is clear that for each $t\geq 0$, the set of labels reached  in $\Th\Psi_\kappa(t)$ and the set of vertices reached by time $t$ in $G(n,\kappa)$ are equal in distribution. So we arrive at:
\begin{lemma}[FPP on $G(n,\kappa)$ is thinned CTLBP]\label{lemma::CTMBPEmbeddinIntoIHRG}
For any fixed $n\geq1$, consider $\Th\Psi_\kappa$ and $G(n,\kappa)$ as defined above. Then for any $i_0\in [n]$, the weight $\CP_n(i_0,j)$ and the hopcount $\CH_n(i_0,j)$ of the shortest weight path between vertices $i_0, j\in [n]$ in $G(n,\kappa)$ is equal in distribution to the  weight and hopcount of the shortest weight path between the root $i_0$ and $j$ in $\Th\Psi_\kappa$.
\end{lemma}
\begin{remark}
We did not use that the edge weights are exponentially distributed. So Lemma \ref{lemma::CTMBPEmbeddinIntoIHRG} holds for i.i.d. edge weights with arbitrary continuous distribution supported on $(0,\infty)$.
\end{remark}

To make the intuitive picture of colliding flows of fluid precise we formally introduce the notion of shortest-weight trees $SWT_k$, for $k\geq1$. Since we cannot follow the progress continuously in time, we keep track of the flows at each split time $\tau_k$. With a slight misuse of notation, let $\CD(k)$ and $\CA(k)$ stand for the collection of dead and alive labels of the vertices that the flow reaches up to and including time $\tau_k$ (as a list, with multiple occurrences). Clearly $|\CA(k)|= S_k$. Let $SWT_0=(\{i_0\}, \tau_0=0)$, and define
\begin{equation}\label{eq::DefOfSWT_k}
SWT_k= \left( \CD(k), \CA(k), \{\tau_0, \tau_1, \ldots, \tau_k\}\right), \: k\geq1.
\end{equation}
The CTLBP $\Psi_\kappa(t)$ can be uniquely reconstructed from the sequence $(SWT_k)_{k=1}^{\infty}$. Note that $SWT_k$ contains all the labels in $\Psi_\kappa$, also the thinned labels and possibly some multiple labels among alive vertices. When we investigate the collision of the two flows from $x$ and from $y$ later, we want to avoid the case that this connection happens at thinned vertices in the BP-s. Thus, we will need an upper bound on the proportion of thinned alive labels of a given type.

\begin{lemma}[Expected number of thinned alive labels]\label{lemma::ThinnedMarksWet}
Fix $k\geq1$ and denote by $\CA^t(k)$ and $\mathrm{th}\CA^t(k)$ the number of alive and thinned alive particles of type $t\in \CS$ after the $k$-th split in the CTLBP. Then under Assumption \ref{ass::sorosszeg_la}
\begin{equation}\label{eq::ExpectedThinnedWet}
\expect{\frac{\mathrm{th}\CA^t(k)}{\CA^t(k)}}\leq \dfrac{\wla+1}{\wla}\dfrac{k}{n}(1+o(1)).
\end{equation}
\end{lemma}
\begin{proof}
We calculate the number of alive thinned vertices of type-$t$ by checking whether the particle that splits at time $\tau_j$ is thinned, then see how many type-$t$ alive descendants it has in its subtree at the $k$-th split. Denoting these descendants by $\CA^{i_j\to t}_j(k)$ and the type of $i_j$ by $t(i_j)$ we get that
\begin{equation}\label{eq::thA^t(k)}
\mathrm{th}\CA^t(k)=\sum_{j=1}^k\;\sum_{s\in\CS} |\CA^{i_j\to t}_j(k)|\indd{i_j \text{ is thinned}|t(i_j)=s}\indd{t(i_j)=s}.
\end{equation}
Let us further introduce
\[
\CA_j^{u\to t}(k) = \{ v \in \CA^t(k): v \text{ is a descendant of $w$ with } w\in \CA^u(j+1) \}.
\]
For $|\CA^{i_j\to t}_j(k)|$ we can argue that if the $j$-th particle to split was of type-$s$, it had $\eta_{su}$ type-$u$ children, then by symmetry and by the memoryless property of the lifetimes we have that
\[ \condexpect{\CA^{i_j\to t}_j(k)}{\eta_{su}, \CA^u(j), \CA_j^{u\to t}(k)}= \sum_{u\in\CS} \frac{\eta_{su}}{|\CA^u(j+1)|} |\CA_j^{u\to t}(k)|.
\]
Combining this and \eqref{eq::thA^t(k)} with the fact that the event that $i_j$ is thinned and $\CA^{i_j\to t}(k)$ are conditionally independent yields that the expectation in \eqref{eq::ExpectedThinnedWet} can be bounded from above by
\begin{equation*}
\sum_{j=1}^k\;\sum_{s,u\in\CS} \expect{\frac{\eta_{su}}{|\CA^u(j+1)|} \frac{|\CA_j^{u\to t}(k)|}{|\CA^t(k)|} } \underbrace{\condprob{i_j \text{ is thinned}}{t(i_j)=s}}_{(\ast)}\underbrace{\prob{t(i_j)=s}}_{(\diamond)}.
\end{equation*}

Recall that $N_k^{(t)}$ denotes the number of splits of type-$t$ vertices among the first $k$ splits and there are $n_t$ different marks that correspond to vertices of type-$t$. The conditional probability in $(\ast)$ can be bounded simply by $N_k^s/n_s$, while $(\diamond)$ equals $S_j^t/S_j$. Thus
\begin{align*}
\expect{\frac{\mathrm{th}\CA^t(k)}{\CA^t(k)}} &\leq \sum_{j=1}^k\;\sum_{s,u\in\CS} \expect{\frac{\eta_{su}}{S_j^u} \frac{|\CA_j^{u\to t}(k)|}{|\CA^t(k)|} } \frac{N_j^s}{n_s} \frac{S_j^t}{S_j} \\
&= \sum_{j=1}^k\;\sum_{s,u\in\CS}\frac{\kappa(s,u)\mu_u}{\wla(j+1)\pi_u}\frac{|\CA_j^{u\to t}(k)|}{|\CA^t(k)|} \frac{\pi_s^2}{\mu_s}\frac{j}{n}(1+o(1)),
\end{align*}
where we used \eqref{eq::n_t/n_to_mu}, \eqref{eq::S_m^t/m}, \eqref{eq::N_m^t/m} and \eqref{eq::S_m/m}. Under Assumption~\ref{ass::sorosszeg_la} $\pi=\mu$ (see \eqref{eq::pi==mu}) and using the symmetry of $\kappa$ the above expression simplifies to
\[
\frac{1}{\wla}\sum_{j=1}^k \frac{1}{n}\sum_{u\in\CS}\frac{|\CA_j^{u\to t}(k)|}{|\CA^t(k)|} \underbrace{\sum_{s\in\CS}\kappa(u,s)\mu_s}_{\wla+1}(1+o(1)) =\dfrac{\wla+1}{\wla}\dfrac{k}{n}(1+o(1)).
\]
\end{proof}

We show a similar result to the assertion of Lemma \ref{lemma::ThinnedMarksWet} for the number of multiple labels among $\CA(k)$.  We need to guarantee that the number of different labels in the set $\CA_k$ is approximately the same as the size of the set, $S_k$. More precisely,

\begin{lemma}[Expected number of multiple labels]\label{lemma::ThinnedMarksAlive}
For all $k\geq1$ and $t\in\CS$, the number of different labels of alive vertices after the $k$-th split
\begin{equation}\label{eq::multiple_labels}
|\CA^t(k)| = S_{k}^t \left(1- \frac{\wla\pi_t}{2\mu_t}\frac{k}{ n}\right).
\end{equation}
\end{lemma}
\begin{proof}
When assigning labels to new vertices of type $t$ we prescribed some constraints on the set $[n]^{(t)}$ from which we choose its label. We can dominate this by repeatedly choosing from $[n]^{(t)}$ without any constraints. From \eqref{eq::S_m^t/m} we know that $S^{t}_{k}=\wla\pi_t k(1+o(1))$. So let us sample $\wla\pi_t k(1+o(1))$ marks from $[n]^{(t)}$ with replacement. For $i\in [n]^{(t)}$ let $X_i$ be the number of times $i$ was chosen, thus $X_i\sim\text{Bin}(\wla\pi_t k(1+o(1)),\,1/n_t)$. The probability that $i$ is chosen at least twice is
\[\prob{X_i\geq2} =1- \prob{X_i=0}-\prob{X_i=1} = \dfrac{\wla^2\pi_t^2}{2}\dfrac{k^2}{n_t^2} + O\left(\dfrac{k^3}{n_t^3}\right).\]
Thus for the expected number of type-$t$ multiple labels
\[ \expect{\sum_{i\in [n]^{(t)}} \indd{X_i\geq2} } = \sum_{i\in [n]^{(t)}} \prob{X_i\geq2} = \frac{\wla^2\pi_t^2}{2\mu_t} \frac{k^2}{n}+o(1). \]
The claim immediately follows since the number of different labels equals the number of alive vertices $S_k^t=\wla\pi_t k(1+o(1))$ minus the multiple labels.
\end{proof}

\section{Connection time}\label{sec::ConnectionTime}

In this section we rigorously examine the intuitive picture of colliding flows. For technical reasons we do not let the fluids flow simultaneously from both vertices, rather we let the fluid flow from $x$ until it reaches some $a_n=o(n)$ vertices, then we "freeze" it, and  start a flow from $y$ until the random time of connection, i.e. when the two flows collide.

The exploration process from $x$ until the split time $\tau_{a_n}^x$ is coded in $SWT_{a_n}^{x}$ (see \eqref{eq::DefOfSWT_k}). Afterwards, the flow from $y$ can only connect to the flow from $x$ via an alive vertex in $SWT_{a_n}^{x}$. So when assigning the labels to the vertices in the BP from $y$ we must leave out the labels $\CD^x(a_n)$ from the possible labels $[n]$. A possible collision edge appears when a label from $\CA^x(a_n)$ appears among the labels in $\CD^y(k)$. The first possible collision edge appears at split $C_n^{(1)}=\min \{ k\geq0: \; \CA^x(a_n) \cap \CD^y(k) \neq \emptyset\}$ at time $\tau_{C_n^{(1)}}^y$. The $i$-th appears at split
\begin{equation}\label{eq::PossibleCollisionTime}
C_n^{(i)}= \min \{ k\geq C_n^{(i-1)}: \; |\CA^x(a_n) \cap \CD^y(k)|=i\}, \text{ at time } \tau_{C_n^{(i)}}^y.
\end{equation}
Thus the weight of a path between $x$ and $y$ is $\tau_{a_n}^x+\tau_{C_n^{(i)}}^y+E_i$, where $E_i$ is the remaining lifetime of the possible collision edge after time $\tau_{a_n}^x$. From the memoryless property of the weights it follows that $E_i\stackrel{d}{=}\mathrm{Exp}(1)$. So the actual connection happens through the possible collision edge that minimizes the expression $\tau_{C_n^{(i)}}^y+E_{i}$. Thus the shortest weight path equals
\begin{equation}\label{eq::ConnectionTime}
\CP_n = \tau_{a_n}^x + \min_{i} \left\{ \tau_{C_n^{(i)}}^y + E_i\right\}.
\end{equation}
Let us denote the split which minimizes the above expression by $C_n^{\text{con}}$. The figure below illustrates the connection time.

\begin{center}
\begin{tikzpicture}
[scale=0.8,
 elo/.style={circle, shade, ball color=white, scale=1.3},
 halott/.style={circle, shade, ball color=black!40, scale=1.3},
 csucs/.style={fill=red!30, circle},
 el/.style ={thick, double = black, double distance = 0.5pt}]

\fill[blue!20] (0,-3.5) arc (-90:90:4.5cm and 3.5cm) node[left,black] {$\tau_{a_n}$};
\filldraw[blue!20] (10,3.3) node[right,black] {$\tau_{C_n^{\text{con}}}$} arc (90:270:4.3cm and 3.3cm);
\filldraw[blue!40] (10,2.8) node[right,black] {$\tau_{C_n^{(1)}}$} arc (90:270:3.5cm and 2.8cm);

\path[el,draw=red] (0,0) node[csucs] (x) {$x$} -- (3.2,0.6) node[halott] {} -- (5.8,0.3) node[elo] {} -- (10,0) node[csucs] (y) {$y$};

\foreach \x/\y/\q/\w/\n in {1.4/2.4/3.5/2.9/1,3.8/1.9/4.6/1.6/2,2/-1.8/5/-2.1/3}
  \draw (\x,\y) node[halott] (\n) {} -- (\q,\w) node[elo] {};

\draw (2) -- (5.4,2.4) node[elo] {};
\draw[el, purple] (3.5,-1) node[halott] {} -- (6.9,-1.2) node[elo] {};
\draw (3) -- (4.2,-2.6) node[elo] {};

\node[below] at (5,-3.4) {{\Large Connection time}};
\node[elo] at (11,2) {};
\node[right, text width=1.8cm] at (11.3,2) {vertices in $\CA^x(a_n)$};
\node[halott] at (11,0.75) {};
\node[right, text width=1.8cm] at (11.3,0.75) {vertices in $\CD^x(a_n)$};
\draw[el, red] (10.8,-0.75) -- (11.2,-0.75) node[right, black, text  width=1.8cm] {optimal path};
\draw[el, purple] (10.8,-2) -- (11.2,-2) node[right, black, text  width=1.9cm] {first collision edge};

\draw (0,4) -- (10,4);
\foreach \x in {0,4.5,5.7,10}
  \filldraw (\x-0.05,3.75) rectangle (\x+0.05,4.25);
\foreach \szoveg/\z in {$SWT^{(x)}$/2.25,$SWT^{(y)}$/7.85}
  \node[above] at (\z,4) {\szoveg};

\end{tikzpicture}
\end{center}

To be able to determine the distribution of the minimum in \eqref{eq::ConnectionTime}, we need a handle on the size of $C_n^{(i)}$. The following proposition, roughly speaking, states that all the possible collision edges appear at $O(n/a_n)$ time with some random constant.
\begin{proposition}[PPP limit of collision edges]\label{prop::WeakConvOfCn}
Denote a homogeneous Poisson Point Process with intensity $\lambda$ by $\mathrm{PPP}(\lambda)$ and let $\hla=\wla \sum_{s\in\CS} \pi_s^2/\mu_s$. Conditioned on the event that both CTLBPs survive, the point process
\begin{equation*}
\left\{ \frac{C_n^{(i)}a_n}{n} \right\}_i \toindis \mathrm{PPP}(\hla) \;\; \text{as } n\to \infty,
\end{equation*}
where $\hla$ simplifies to rate $\wla$ under Assumption \ref{ass::sorosszeg_la}.
\end{proposition}
\begin{proof}
We first show by induction that for fix $n$
\begin{align}\label{eq::DistTimeofCollEdges}
\frac{C_n^{(i)}a_n}{n} &\stackrel{d}{=} \mathrm{Gamma}(i, \hla_n), \text{ where} \\
\hla_n &=\frac{n}{a_n} \sum_{s\in\CS} \frac{S_{a_n}^{x,s}}{n_s}\pi_s(1+o(1))+o(1). \nonumber
\end{align}
For $C_n^{(1)}$ we can write
\begin{equation*}\label{eq::P(C_n>m)}
\prob{C_n^{(1)}>x\frac{n}{a_n}} = \prod_{j=1}^{xn/a_n} \expect{\condprob{C_n^{(1)}>j}{C_n^{(1)}>j-1, \CF_{j-1}}},
\end{equation*}
where $\CF_{j-1}$ is the $\sigma$-algebra generated by $SWT^{x}_{a_n}$ and $SWT^{y}_{j-1}$. To calculate $\condprob{C_n>j}{C_n>j-1, \CF_{j-1}}$ we have to sum over the types in $SWT^{y}$ and find the probability that it does not connect to an alive vertex in $SWT^{x}$ of the same type. Below we use the result of Lemma~\ref{lemma::ThinnedMarksAlive}, and we substitute the number of different marks in $\CA^{x,t}(a_n)$ by $S_{a_n}^{x,t}$, and neglect the error factor of order $(1-(\wla a_n)/n)$ along the lines. This error can be included in the $o(1)$ term of the last line of the display below.
\begin{multline*}
\begin{aligned}
\prob{C_n^{(1)}> \frac{x n}{a_n}}
 &= \prod_{j=1}^{x n/a_n} \left[\sum_{t\in\CS} \frac{S_{j-1}^{y,t}}{S_{j-1}^{y}}\left(1-\frac{S_{a_n}^{x,t}}{n_t}\right)\right]
 = \prod_{j=1}^{x n/a_n} \left[1-\sum_{t\in\CS} \frac{S_{j-1}^{y,t}}{S_{j-1}^{y}}\frac{S_{a_n}^{x,t}}{n_t}\right] \\
&= \exp \bigg(- x\frac{n}{a_n}\sum_{t\in\CS}\frac{S_{a_n}^{x,t}}{n_t} \cdot \underbrace{\frac{a_n}{xn} \sum_{j=1}^{xn/a_n}\frac{S_{j-1}^{y,t}}{S_{j-1}^{y}}}_{(\ast)} +o(1)\bigg),
\end{aligned}
\end{multline*}
where $(\ast)$ equals $\pi_t(1+o(1))$ by \eqref{eq::S_m/m}. So $\frac{C_n^{(1)}a_n}{n}\stackrel{d}{=}\mathrm{Exp(\hla_n)}=\mathrm{Gamma}(1,\hla_n)$. From the induction hypothesis
\begin{multline*}
\begin{aligned}
&\prob{C_n^{(i+1)}>\frac{x n}{a_n}} \\
&= \int_0^x \condprob{C_n^{(i+1)}>\frac{x n}{a_n}}
{\frac{C_n^{(i)}a_n}{n}=s+o(1) }\frac{\hla_n^is^{i-1}}{(i-1)!}e^{-\hla_ns}\mathrm{d}s +\prob{C_n^{(i)}>\frac{xn}{a_n}} \\
&= \int_0^x \prod_{j=sn/a_n}^{xn/a_n} \left[1-\sum_{t\in\CS} \frac{S_{j-1}^{y,t}}{S_{j-1}^{y}}\frac{S_{a_n}^{x,t}}{n_t}\right] \frac{\hla_n^is^{i-1}}{(i-1)!}e^{-\hla_ns}\mathrm{d}s +\prob{C_n^{(i)}>\frac{xn}{a_n}} \\
&= \frac{\hla_n^ix^{i}}{i!}e^{-\hla_nx}+\prob{C_n^{(i)}>\frac{xn}{a_n}}.
\end{aligned}
\end{multline*}
Differentiating the cdf of $C_n^{(i+1)}a_n/n$ with respect to $x$ yields the pdf of $\mathrm{Gamma}(i+1,\hla_n)$, which proves \eqref{eq::DistTimeofCollEdges}. Thus for fixed $n$ the point process
\begin{equation*}
\bigg\{ \frac{C_n^{(i)}a_n}{n} \bigg\}_i \text{ is a } \mathrm{PPP}(\hla_n).
\end{equation*}
From \eqref{eq::n_t/n_to_mu} and \eqref{eq::S_m^t/m} it follows that $\hla_n\toas\hla$ as $n\to\infty$. The weak convergence now immediately follows.
\end{proof}

This result shows that when finding the shortest path it makes no difference to let the fluids flow simultaneously or to delay one of them. When they flow simultaneously, the number of vertices explored by both flows is of order $\sqrt n$ which is optimal in the sense that in every other case the explored vertices are of larger magnitude.

We will see in Subsection \ref{subsec::prooffinite} that to determine the distribution of $\tau_{C_n^{\text{con}}}+E_{\text{con}}$ we need the following.

\begin{lemma}\label{lemma::gumbeldistribution}
Let $(P_i)_i$ denote the points of a $PPP(1)$ process and independent of that, $E_i\stackrel{d}{=}\mathrm{Exp}(1)$, also independent. Then
\begin{equation*}
\min_i \left\{\frac{1}{\wla}\log P_i +E_i\right\} \stackrel{d}{=} -\frac{1}{\wla}X+\frac{1}{\wla}\log (\wla+1),
\end{equation*}
where $X$ follows a standard Gumbel distribution, i.e. $\prob{X\leq x}=e^{e^{-x}}$.
\end{lemma}
\begin{proof}
For convenience let $X_i \sim \frac{1}{\wla}\log P_i$. We will calculate the tail distribution of the minimum by conditioning on the Poisson points first:
\[ \ba  \Pv \left[ \min_i X_i + E_i \ge z \right] &= \Ev\left[ \Pv \left[ \forall i  \ X_i + E_i > z | X_1, X_2 \dots \right] \right]\\
&= \Ev \left[ \prod_{i} e^{-(z-X_i)_+} \right] \\
&= \Ev\left[\prod_{i: P_i < e^{\wla z}} e^{-z}  (P_i)^{\frac{1}{\wla}}\right]
 \ea \]

Now, the number of Poisson points $Z$ in the interval $[0,e^{\wla z}]$ follows a Poisson random variable with parameter $e^{\wla z}$.
Conditioning on this number $Z$, the points $P_i, i\le Z$ are independent and uniform in the interval $[0,e^{\wla z}]$. Thus, we can calculate the expected value of the product on the right hand side of the previous display as follows:
\[ \ba  \Ev\left[\prod_{i: P_i < e^{\wla z}} e^{-z} (P_i)^{\frac{1}{\wla}}\right]&= \Ev\left[ \Ev\left[e^{-z} U_i^{\frac{1}{\wla}}\right]^Z\right] \\
&= \Ev\left[\left(\frac{\wla}{\wla+1}\right)^Z\right]= \exp\left\{e^{ \wla z}\left(1-\frac{\wla}{\wla+1}\right)\right\}.\ea \]

It is easy to see that if $X$ is a standard Gumbel random variable with $\Pv(X\le x) = e^{-e^{-x}}$, then $\Pv(-a X + b > x) = e^{-e^{x/a} e^{-b/a}}$. Thus, here with $a = \frac{1}{\wla}$ and $b=\frac{1}{\wla}\log (\wla+1)$ we get the claim.
\end{proof}

We finish the section with a result that shows that the index $i$ where the minimum is taken is stochastically bounded by a Geometric distribution of parameter $1/(\wla+1)$.
\begin{lemma} \label{lem::argminlemma}
The probability that the shortest weight path is not among the first $k$ collision edges is decaying exponentially in $k$, i.e.
\[ \Pv\left[\arg (C_n^{(i)} = C_n^{\text{con}}) > k \right] \le \left(\frac{\wla}{\wla+1}\right)^k.\]
\end{lemma}
As a consequence of the lemma we immediately get that the distribution of the rescaled connection time is stochastically dominated by a sum of independent exponentials with parameter $i$ up to a geometric random variable independent of them, namely
\be\label{eq::geo_dominance} P_n^{\text{con}}:=\frac{a_n C_n^{\text{con}}}{n}  \le \sum_{i=1}^{N} \tilde E_i, \ee
where $N\sim$ Geo$(\tfrac{1}{\wla+1})$ and independently $\tilde E_i$-s are independent Exp$(1)$-s.

\begin{proof}[Proof of Lemma \ref{lem::argminlemma}]
 We use again the notation $X_i \sim \frac{1}{\wla}\log ( P_i)$, where $P_i$ is the $i$-th point in a PPP( $1$) process. To show that the minimum is taken at an index at least $k+1$, we will condition on the value of the minimum (=z) and also on the value of the $X_{k+1}= c$ with $z \ge c$. Thus we have
\[ \ba \Pv\left[\arg\min > k \right] & = \Pv\left[\min_{i\le k} (X_i +E_i) > \min_{j\ge k+1} (X_j + E_j) \right] \\
&=\Ev \left[ \Pv\left[ \forall i \le k, \ E_i > z-X_i \left| X_{k+1}=c, \min_{j\ge k+1} (X_j + E_j) =z \right. \right] \right] \\
&=\Ev \left[\Ev \left[\prod_{i\le k} e^{-(z-X_i)}| X_{k+1} = c, \min = z \right]\right].
\ea
\]

Now $X_{k+1}=c$ means that the $k+1$-th point in the Poisson process  is $P_{k+1}=e^{\wla c}$. Conditioning on this information means that the first $k$ points have the same distribution as $U_i, i=1,\dots, k$ independent uniform points on $[0,e^{\wla c}]$. Thus, the expectation can be calculated as follows:
\[ \Ev \left[\prod_{i\le k} e^{-(z-X_i)}| X_{k+1}=c \right] = \prod_{i\le k} \Ev \left[ e^{-z} U_i^{\frac{1}{\wla}}\right]= \left(e^{c-z}\frac{\wla}{\wla+1}\right)^k. \]
Then clearly we have
\[ \ba\Pv\left[\arg\min > k \right] &= \left(\frac{\wla}{\wla+1}\right)^k \Ev\left[ \Ev\left[e^{k (c-z)} | X_{j+1}=c, \min_{j\ge k+1}=z\right]\right]\\
&\le \left(\frac{\wla}{\wla+1}\right)^k,
\ea \]
where the last inequality comes from the fact that $z-c\ge 0$ almost surely. This is trivial since the sequence of $X_j$-s are increasing, $X_j \ge X_{k+1}$ for all $j \ge k+1$. Thus, $z=\min\limits_{j\ge k+1}{X_j+E_j} \ge X_{k+1}=c$.
\end{proof}

\section{Proof of main results}\label{sec::proof}

We begin the section with the proof of Theorem \ref{thm::GeneralSetting} for finite-type graphs. We continue with the discussion of approximating kernels and then prove Theorem \ref{thm::GeneralSetting} in the general setting. The section is concluded with the proof of Theorem \ref{thm::DenseSettingMain}. The proofs are analogous to the ones in \cite{RvdH_FPP}. The idea of using approximating kernels comes from \cite[Section 7]{BB_IHRGM}.

\subsection{Proof of Theorem \ref{thm::GeneralSetting}: finite-type setting}\label{subsec::prooffinite}

Let $(\CS,\mu)$ be an arbitrary finite-type ground space and $\kappa$ a kernel that satisfies Assumption \ref{ass::sorosszeg_la}. We first argue that the probability that the shortest weight path contains thinned vertices -i.e. it is not a real shortest path -  is $o(1)$. Denote the split of a type-$t$ vertex by $t\dagger$, the event that the connection happens by $\mathrm{con}$ and recall that $\CA^{x,t}(k)$ is the collection of labels of alive type-$t$ vertices after the $k$-th split in the flow of $x$.

From a simple union bound, the probability that the connection happens through a thinned alive vertex $v$ can be bounded from above by
\begin{equation*}
\sum_{t\in\CS} \left(\prob{v\in\mathrm{th}\CA^{x,t}(a_n)}+ \prob{v\in\mathrm{th}\CA^{y,t}(C_n^{\text{con}})}\right)\prob{\mathrm{con}, t\dagger}.
\end{equation*}
Using Lemma \ref{lemma::ThinnedMarksWet} with $k=a_n$ and $k=C_n^{\text{con}}$ respectively yields that this term equals
\begin{align*}
&\sum_{t\in\CS}\frac{S_{C_n^{\text{con}}}^{y,t}}{S_{C_n^{\text{con}}}} \left(\frac{\wla+1}{\wla}\left(\frac{a_n}{n}+\frac{C_n^{\text{con}}}{n-a_n}\right)(1+o(1))\right) \\
&= \frac{\wla+1}{\wla} \left(\frac{a_n}{n}+\frac{n}{n-a_n}\frac{P_n^{\text{con}}}{a_n}\right)(1+o(1)) \to 0\; \text{ as } n\to\infty.
\end{align*}
We divide in the flow of $y$ by $n-a_n$ since the labels are chosen from $[n]\setminus\CD^x(a_n)$. Further note that by Lemma \ref{lemma::ThinnedMarksAlive}, in the proof of Lemma \ref{lemma::ThinnedMarksWet} $|\CA^{x,t}(a_n)|$ can be replaced by $S_{a_n}^{x,t}= \pi_t a_n (1+ o(1))$.

Then, to get the second line we used that $a_n C_n^{\text{con}}/n\sim P_n^{\text{con}}$,  is tight by \eqref{eq::geo_dominance}. This shows that if $a_n / n$ and $1/a_n$ are both tending to zero, - i.e. for all  $a_n=o(n)$, the shortest weight path whp. does not contain a thinned vertex.

\vspace{3mm}

Now we turn to determine the distribution of the shortest weight path. We know from \eqref{eq::ConnectionTime} that
\be\label{eq::minpath} \CP_n = \tau_{a_n}^x + \min_{i} \left\{ \tau_{C_n^{(i)}}^y + E_i\right\}. \ee
In $\Psi_\kappa(t)$, the rescaled total number of alive individuals,  $M_t=e^{\wla t}|\un S(t)|$ is a martingale. Thus, $\tau_k$  can be expressed as
\be \label{eq::splittime} \tau_k= -\frac{1}{\wla}\log \frac{M_{\tau_k}}{\wla} + \frac{1}{\wla}\log \frac{S_k}{\wla k} + \frac{1}{\wla} \log k.\ee
Applying this formula to the minimum in \eqref{eq::minpath} we have
\[ \ba \min_{i} \left\{ \tau_{C_n^{(i)}} +E_i \right \} &= \min_i \Big\{ -\frac{1}{\wla}\log \left( \frac{\hla M^y_{\tau_{C_n^{(i)}}}}{\wla} \right)+ \frac{1}{\wla}\log \left(\frac{S^y_{C_n^{(i)}}}{\wla C_n^{(i)}} \right) \\
& +  \frac{1}{\wla} \log\left(  \hla C_n^{(i)}\right)+ E_i \Big\}. \ea \]

For $i$ fixed, $n\to \infty$, $\hla C_n^{(i)}= = \frac{n}{a_n} \hla P_i \to \infty$ and thus conditioned on survival of the branching process, $\tau_{C_n^{(i)}}\to \infty$ holds as well, implying $M^y_{\tau_{C_n^{(i)}}}\to (W^y| W^y>0):=\hat W^y$ a.s. and in $L_2$, and $S^y_{C_n^{(i)}}/(\wla C_n^{(i)}) \to 1$ also a.s.. Further, we also know from Lemma \ref{prop::WeakConvOfCn} that the law of $(\hla P_i)_i$ converges to a $PPP(1)$ process.
Thus, the minimum becomes asymptotically as $n\to \infty$
\[ \min_{i} \{ \tau_{C_n^{(i)}} +E_i \}  = - \frac{1}{\wla}\log \frac{\hat W^y \hla }{\wla} + \frac{1}{\wla} \log \frac{n}{a_n}  + \min_i \left\{ \frac{1}{\wla} \log (\hla P_i) + E_i  \right\}.
 \]
For the last term we can apply Lemma \ref{lemma::gumbeldistribution} to get
\be \label{eq::mini} \min_i \{ \tau_{C_n^{(i)}}+E_i\} = -\frac{1}{\wla} \log \frac{\hat W^y \hla }{\wla(\wla+1)} + \frac{1}{\wla} \log \frac{n}{a_n} - \frac{1}{\wla} X.\ee
with $X$ denoting a standard Gumbel random variable. We can also use that under Assumption \ref{ass::sorosszeg_la} $\hla = \wla$. Further, applying \eqref{eq::splittime} and \eqref{eq::mini} to the expressions in \eqref{eq::minpath}, we arrive at
\[ \mathcal P_n = \frac{1}{\wla} \log n - \frac{1}{\wla}\log \hat W^{x} \hat W ^{y} - \frac{1}{\wla} X+\frac1\wla\log\left( \wla (\wla+1)\right), \]
with $X$ being a standard Gumbel random variable, $\hat W ^{z}, z=x,y$ denoting the limiting random variable of the (independent) martingales $e^{-\wla t} |S(t)^{z}|$ in $\Psi_\kappa^z$, $z=x,y$ conditioned on non-extinction.  We know that $W^i >0$ a.s. on non-extinction of the processes, so these quantities are well-defined.

\vspace{1mm}
Now we turn to the derivation of the limit theorem for the hopcount $\CH_n$.
We start first proving that the hopcount of the connecting vertices in the processes $\Psi_\kappa^x$ and $\Psi_\kappa^y$ are independent conditioned on their types. We remind the reader that first the flow of $x$ is constructed, which is then frozen at time $a_n$. $(\CA^{x,t}(a_n), \CD^{x,t}(a_n))$ denotes the labels of the alive and dead vertices of type $t$, respectively.
 Now the evolution of the flow of $y$ conditioned on $SWT^{x}(a_n)$ is as follows: leaving out the labels in $\CD^{x}(a_n)$, each time  a vertex splits in $\Psi_\kappa^{y}$, the label of its type-$t$ children are picked uniformly at random without replacement amongst the possible labels of the corresponding type.
\begin{lemma} \label{lem::independence} We have the following statements given the event of $$\{\text{Collision happens from } SWT^{(y)} \text{ to } SWT^{(x)} \text{ at the } C_n\text{-th split at a type } t \text{ vertex } \},$$
\begin{enumerate}
\item the label $v_C$ at which this happens is uniform among all labels in $\mathcal A^{x,\, t}_{a_n}$.
\item Further, the hopcounts $G^{x,t}_{a_n}$, $G_{C_n}^{y,t}$ are independent given that the collision happens at a type $t$ vertex.
\item The distribution of $G^{x,t}_{a_n}$ and $G_{C_n}^{y,t}$ is the same as of a uniformly picked type $t$ alive individual in the processes $\CA^{x}_{a_n}$ and $\CA^{y}_{C_n}$ .
\end{enumerate}
\end{lemma}

\begin{proof}
 The first two statement of the Lemma is a straightforward consequence of the following urn-problem: In an urn there are $M$ balls of type $A$ (alive) and $N$ balls of type $U$ (untouched), each of them labeled. We do the following procedure: in the $k$-th step we draw $d_k$ balls without replacement, add the label of type $A$ and type $U$ balls to sets $L_A$ and $L_U$, respectively, and then put them back top the urn. Thus, $L_A(k)$ consists of all the labels of type-$A$ balls which has been drawn before or at step $k$.
 It is easy to show that at any time, the content of the set $L_A$ and $L_U$ is a uniformly picked set of size $|L_A|$ and $|L_B|$ among all the labels in $A$ and $B$, respectively. In particular, for every label $v \in A$ we have
\[ \Pv[v\notin L_A(k) ] = \prod_{j\le k} \left(1-\frac{d_j}{M+N}\right).  \]
Now let $A=\CA^{x,t}_{a_n}$ and $U=[n_t]\setminus(\CD^{x,t}_{a_n}\cup\CA^{x,t}_{a_n}$. Then, $L_A(k) \cup L_U(k)= \CA^{y,t}(k)$. The previous argument says that at any time, the labels in $L_A= \CA^{y,t}\cap \CA^{x,t}$ are uniformly picked from the labels of $\CA^{x,t}$.
The possible collision edges between the processes $x$ and $y$ are established such that in each step $k$ with some probability $p_k$ (which is the probability that the next dying particle in $\Psi_\kappa^y$ is of type $t$), we pick a uniform label among $L_A(k) \cup L_U(k)$, and check if it is of type $A$. Clearly, whenever this holds true, a possible collision edge is formed in the two shortest weight trees. Also, it is clear from the previous argument that conditioned on the picked label to be in $L_A(k)$, the label of it is a uniformly picked label among $\CA^{x,t}_{a_n}$, (and clearly also uniform in $L_A(k)$). Further, the step $k$ when a label of type $A$ enters $L_A$ is independent of the label itself, thus the generation of the label at the connection in $SWT^{x}$ and in $SWT^{y}$ are independent and equal to the generation of a uniformly picked alive individual of type $t$.
\end{proof}

Now we are ready to determine the limit distribution of the hopcount.
Let $G_{k}^{z,t}$ denote the generation of a uniformly picked alive individual of type $t$ in $\CA^z(k)$, $z= x,y$, and recall the definition of  $C_n^{\text{con}}$.
Then we have
\[ \CH_n = \sum_{t\in \CS} \ind\{C_n^{\text{con}} \cap t\ \dag\} \left(  G^{x,t}_{a_n} +G_{C_n^{\text{con}}}^{y,t} \right).  \]
Thus
\[ \ba \frac{\CH_n - \frac{\wla+1}{\wla} \log n}{\sqrt{\frac{\wla+1}{\wla} \log n}} & = \sum_{t \in S} \ind\{ C_n^{\text{con}} \cap t \, \dag \} \frac{ G^{x,t}_{a_n} - \frac{\wla+1}{\wla} \log a_n}{\sqrt{\frac{\wla+1}{\wla} \log a_n}} \cdot \sqrt{\frac{\log a_n}{\log n}} \\
&+ \sum_{t\in \CS} \ind \{ C_n^{\text{con}} \cap t\dag \} \frac{ G^{y,t}_{C_n^{\text{con}}} - \frac{\wla+1}{\wla} \log C_n^{\text{con}} }{\sqrt{\frac{\wla+1}{\wla} \log C_n^{\text{con}}}} \cdot \sqrt{\frac{\log C_n^{\text{con}}}{\log n}}\\
& + \frac{\frac{\wla+1}{\wla}\log\left(\frac{C_n^{\text{con}} a_n}{n}\right) }{ \sqrt{ \frac{\wla+1}{\wla} \log n}}.
\ea \]
First use that conditioned on $C_n^{\text{con}}$ and the type, the two terms containing $G^{z,t}_*$ converge to independent standard normal variables (independently of the type).
 Further, recall the the distributional bound on $P_n^{\text{con}}=C_n^{\text{con}} a_n/n$ in \eqref{eq::geo_dominance} to see that the last term tends to zero. Lemma \ref{lem::independence} ensures the independence of the two limiting normal variables, thus we get the following distributional limit of the right hand side of the last display:
\[   N\left(0, \frac {\log a_n}{\log n}\right) + N\left(0, \frac{\log \left(\frac {n P_n^{\text{con}}}{a_n}\right)}{\log n}\right) \to N(0,1). \]

Combining this with Lemma \ref{lem::argminlemma} and \eqref{eq::geo_dominance} we get that the last term vanishes as $n\to \infty$ and the variance of the normal distribution is also tending to $1$.
\subsection{Approximation of kernels}\label{subsec::ApproxOfKernels}

For a sequence of partitions of $\CS$ a sequence of regular finitary approximating kernels will be defined, each satisfying \eqref{ass:integral_la+1_general}. In the regular finitary case we may assume that the type space $\CS$ is finite (the regular finitary case and the finite-type case differ only in notation). Furthermore, we can assume that $\mu_s>0$ for every $s\in\CS$, however an argument is needed. We cannot simply ignore such types, since measure zero sets can alter $G(n,\kappa)$ significantly. We can argue as follows, \cite{BB_IHRGM}.

Suppose that $\mu_s=0$ for some $s\in\CS$. Start by redefining the kernel $\kappa'(s,t)=\kappa'(t,s):=\max\kappa$ for every $t\in\CS$ and leaving it alone otherwise. Then define a new probability measure $\mu'$ by shifting some small mass $\eta$ over to $s$ from the other types. Clearly $\mu'_t>0$ for every $t\in\CS$. Possibly the types of some of the vertices changes, so change them correspondingly. This way we obtain a vertex space ${\large \nu}'=(\CS,\mu',(\mathbf{x}'_n)_{n\geq1})$ with kernel $\kappa'$. It is not hard to see that we can couple $G(n,\kappa)$ and $G'(n,\kappa')$ so that $G(n,\kappa) \subseteq G'(n,\kappa')$. 
Finally, letting $\eta\to0$, the norm of $T_{\kappa'}$ with respect to $\mu'$ tends to the norm of $T_{\kappa'}$ with respect to $\mu$, which is equal to the norm of $T_{\kappa}$, since $\kappa=\kappa'$ a.e. Iterating for other measure-0 types, we can see that it suffices to consider cases where $\mu_t>0 \; \forall t\in\CS$.

\vspace{1mm}
We continue with the definitions of the approximating kernels. Given a sequence of finite partitions $\alpha_m=\{A_{m1},\ldots,A_{mM_m}\}$, $m\geq1$, of $\CS$ and an $x\in\CS$, we define $i_m(x)$ as the element of $\alpha_m$ in which $x$ falls, formally $x\in A_{m,i_m(x)}$. As usual, $\mathrm{diam}(A)$ denotes $\sup\{d(x,y): x,y\in A\}$ for $A\subset \CS$, where $d$ is the metric on our metric space $\CS$. Lemma 7.1 of \cite{BB_IHRGM} states that for any ground space $(\CS,\mu)$ there exists a sequence of finite partitions of $\CS$ such that \vspace{-2mm}
\begin{enumerate}[1)]
{\setlength\itemindent{1cm}\item each $A_{mi}$ is a $\mu$-continuity set, \label{proper::mu-continuity} }\vspace{-1mm}
{\setlength\itemindent{1cm}\item for each $m$, $\alpha_{m+1}$ refines $\alpha_m$, \label{proper::refine}}\vspace{-1mm}
{\setlength\itemindent{1cm}\item for a.e. $x\in\CS$, $\mathrm{diam}(A_{m,i_m(x)})\to0$, as $m\to\infty$. \label{proper::diamto0}}
\end{enumerate}

For such a sequence of partitions we can define a sequence of approximations of $\kappa$ by taking its average on each $A_{mi}\times A_{mj}$:
\begin{equation}\label{def::aprrox_kernel}
\bar\kappa_m (x,y) := \dfrac{1}{\mu(A_{m,i_m(x)})\cdot \mu(A_{m,i_m(y)})} \iint\limits_{A_{m,i_m(x)}\times A_{m,i_m(y)}} \!\!\!\!\!\!\!\!\!\!\!\!\kappa(s,t)\mathrm{d}\mu(s)\mathrm{d}\mu(t).
\end{equation}
If $\kappa$ is continuous a.e. then property \ref{proper::diamto0}) implies that $\bar\kappa_m (x,y)\to\kappa(x,y)$ for a.e. every $(x,y)\in\CS^2$. To be able to apply our theorems for finite-type kernels we need to guarantee that Assumption \ref{ass::sorosszeg_la} holds for all $\bar\kappa_m$. Thus, considering $\bar\kappa_m$ as a finite-type kernel with respect to the partition $(A_{m1},\ldots,A_{mM_m})$, the row sums of $\bar\kappa_m$ weighted by $\mu$ must be equal to some constant $c_m$. In fact easy calculations show that independently of the partition sequence or the ground space, only using assumption \eqref{ass:integral_la+1_general} and the fact that $\kappa$ is symmetric, this holds with $c_m\equiv\wla+1$.

\subsection{Proof of Theorem \ref{thm::GeneralSetting}: general setting}\label{subsec::proofgeneral}
The extension of the proof for general $(\CS,\mu)$ goes with usual discretization techniques, however, one must be careful with the error terms to maintain the distributional convergence.

Let $(\CS,\mu)$ be an arbitrary ground space and kernel $\kappa$ satisfies the conditions of the theorem. These define the sequence of random graphs $(G(n,\kappa))_{n\geq1}$. Take any sequence of finite partitions $\CP_m=\{A_{m1},\ldots,A_{mM_m}\}$, $m\geq1$, that satisfy properties \ref{proper::mu-continuity}), \ref{proper::refine}) and \ref{proper::diamto0}) described in Subsection \ref{subsec::ApproxOfKernels}. For each $m$, consider the finite type approximating kernel $\bar\kappa_m$ defined in \eqref{def::aprrox_kernel}, with ground space $(\CS_m,\mu)$ (where $|\CS_m|=M_m$). As a result we obtain the sequence $(G(n,\bar\kappa_m))_{n,m\geq1}$. Note that in the proofs for finite type kernels none of the estimates depend on $\mu_t$ or the cardinality of $\CS_m$, so all the error terms are uniform. The condition $\sup \kappa(x,y)<\infty$ is necessary because it is used in the proof of Lemma \ref{lemma::BINPOIcoupling}.

To prove the results we let $n$ and $m$ tend to $\infty$ simultaneously in a carefully chosen way. For fixed $m$, from the proof of \cite[Lemma 2.1 and Theorem 3.1]{Buhler} it is easy to see that
\begin{equation}\label{eq::Wassersteinerror} \left|\prob{\frac{H_n^{m(n)}- \frac{\wla+1}{\wla} \log n}{\sqrt{\frac{\wla+1}{\wla}\log n}} <x} - \Phi(x)\right|  \le C(\wla)\left(\dfrac{1}{\sqrt {\log a_n}}+\dfrac{1}{\sqrt {\log \frac{n}{a_n}}}\right),
\end{equation}
where $C(\wla)$ is a $\wla$-dependent constant. Thus, with the choice $a_n=\sqrt{n}$ we get the error of order $1/\sqrt{\log n}$.

To be able to couple the graphs $G(n,\kappa)$ and $G(n,\bar\kappa_m)$, we need a fine relation between $\kappa(x,y)$ and $\bar\kappa_m(x,y)$. Since $\kappa$ is uniformly continuous, $\exists \varepsilon_m$ s.t. for all $x,y$, and all $(u,v)\in A_{m,i_m(x)}\times A_{m,i_m(y)}: $
\begin{equation*}
\bar\kappa_m(u,v)\leq \kappa(x,y)(1\pm\varepsilon_m)\quad \text{ if } |u-x|<\delta_m \text{ and } |v-y|<\delta_m,
\end{equation*}
where $\mathrm{diam}(A_{m,i_m(x)})<\delta_m$ and $\mathrm{diam}(A_{m,i_m(y)})<\delta_m$. Abbreviate $A_{m,i_m(x)}$ by $A_{mx}$. In $G(n,\kappa)$ the edge-probability between two vertices of types $x$ and $y$ is $\kappa(x,y)/n$, while in $G(n,\bar\kappa_m)$, between types $A_{mx}$ and $A_{my}$ this probability is $\bar\kappa_m(x,y)/n \in \kappa(x,y)(1\pm\varepsilon_m)$. Thus
\begin{equation*}
\prob{\indd{ \{x,y\}\in e(G(n,\kappa))} \neq \indd{ \{A_{mx},A_{my}\}\in e(G(n,\bar\kappa_m))} } \leq \dfrac{2\varepsilon_m}{n}.
\end{equation*}
Summing over all possible edges, we find for the edge sets that
\begin{equation}\label{eq::couplingedgeset}
\prob{ e(G(n,\kappa)) \neq e(G(n,\bar\kappa_m)) } \leq \dfrac{2n^2\varepsilon_m}{2n}=n\varepsilon_m.
\end{equation}
For a fix $m$. The $\delta_m$ and uniform continuity of $\kappa$ defines $\varepsilon_m$. Let
\[ m(n):= \inf\left\{ m :\, \varepsilon_m n \sqrt{\log n}\le 1 \right\}. \]
 Then, for all $m>m(n)$, the coupling between $G(n, \bar\kappa_m)$ and $G(n, \kappa)$ fails only with probability less than $1/\sqrt{\log n} $.
 Under the coupling, also for the hopcount we have \[ \prob{\CH_n\neq \CH^{m(n)}_n}\le 1/\sqrt {\log n} = o(1). \]
Combining this error bound with the one in \eqref{eq::Wassersteinerror} we obtain that
\[ \begin{aligned} &\prob{\frac{\CH_n- \frac{\wla+1}{\wla} \log n}{\sqrt{\frac{\wla+1}{\wla}\log n}} <x} \\
&= \condprob{\frac{\CH^{m(n)}_n- \frac{\wla+1}{\wla} \log n}{\sqrt{\frac{\wla+1}{\wla}\log n}} <x}{\CH_n = \CH^{m(n)}_n}\left(1-o(1)\right) + \prob{\CH_n\neq \CH^{m(n)}_n}\\
&= \Phi(x)\left(1-O\left(\sqrt{\log n\;}^{-1}\right)\right) + C(\wla) /\sqrt {\log n} + 1/\sqrt {\log n}.  \end{aligned} \]
Finally letting $n\to\infty$ (thus $m(n)\to\infty$ also), we obtain the desired result for the hopcount.

\vspace{2mm}

Now we turn to the proof of the convergence of the shortest weight path.
To avoid conflicting notation we will denote $P_n(\kappa)$ the shortest weight path belonging to $G(n, \kappa)$.
We can use the same coupling argument as for the hopcount to get the estimate
\[ \prob{ P_n(\kappa)\neq P_n( \bar \kappa_{m(n)})}\le \frac{1}{\sqrt{\log n}}\]
We know from the finite type case that
\[ P_n( \bar \kappa_{m(n)} )-\frac{1}{\wla}\log n\toindis  -\frac{1}{\wla}\hat W^x_{(m(n))} \hat W^y_{(m(n))}  -\frac1\wla X,\]
where $X$ is a standard Gumbel variable, $\hat W^i_{(m(n))}, i=x,y\ $ is i.i.d. random variables, distributed as the limit of the martingales arising from the branching processes with kernel $\bar\kappa_{m(n)}$, conditioned on being positive.
Since for all $m$, the row sums of $\bar \kappa_m$ equals $\wla$ i.e. each particle has a Poi$(\wla)$ total number of children, it is not hard to see that the limit of the martingales
\[ W_{(m(n))} \buildrel {d}\over{\equiv} W. \]
This finishes the proof of the distributional convergence of the shortest weight path.
\begin{remark}
Besides Lemma \ref{lemma::summable}, this is the other spot where the generalization for $\kappa$ not satisfying Assumption \eqref{ass:integral_la+1_general} would fail: to get the distributional convergence without Assumption \eqref{ass:integral_la+1_general}, we should show that $W_{(m(n))}\to W$. But the relation of the limits of the approximating Branching Process martingales is not clear at this point to us.
\end{remark}
\subsection{Proof of Theorem \ref{thm::DenseSettingMain}: dense setting}\label{subsec::proofdense}

In the dense setting, where $\wla_n\to\infty$, we have a sequence of kernels $\kappa_n, \, n=1,2,\ldots$. The type $t$ neighbors of a type $s$ vertex have distribution $\eta_{st}^{(n)}\stackrel{d}{=}\text{Bin}(n_t-\delta_{st}, \kappa_n(s,t)/n)$. We avoid the coupling with Poisson random variables done in Lemma \ref{lemma::BINPOIcoupling} by immediately applying the CLT result of \cite{Kharlamov1} to $\Psi_{\kappa_n}$ where the offspring distribution $(D_i^{(n)}|\text{type }s \text{ splits})$ is the sum of independent binomial random variables $\eta_{st}^{(n)},\, t\in\CS$.
 We will apply a similar argument than the one in the general case. Namely, we get that in $\Psi_{\kappa_n}$, for a uniformly picked type-$t$ individual at step $k$
 \[ \left|\Pv\left[ \frac{G_{k}^{(n),t}- \frac{\wla_n+1}{\wla_n} \log k}{\sqrt{\frac{\wla_n+1}{\wla}\log k}}<x\right]-\Phi(x) \right|\le C(\wla)\frac{1}{\sqrt{\log k}}, \]
\vspace{3mm}
which, when considering the connection of the flows at $k=a_n$ and $C_n^{\text{con}}= \Theta(n/a_n) $, will yield an error term of $1/\sqrt{\log n}$ for $\CH_n$.
Considering $\wla_n \to \infty$, the term $(\wla_n + 1)/\wla_n \to 1$ in the denominator and this immediately yields the desired result for the hopcount in Theorem \ref{thm::DenseSettingMain}.

The centering constant can be replaced by $\log n$ if and only if $(\frac{\wla_n+1}{\wla_n}-1)\sqrt{\log n}\to 0$, or equivalently $\sqrt{\log n}=o(\wla_n)$.

To prove the part concerning the length shortest-weight path, with the limit taken diagonally, we need to be a bit more careful to determine the distribution of the split times $\tau_{a_n}$ and $\tau_{C_n^{\text{con}}}$. Since the time between two consecutive splits given the number of alive individuals in the BP is just the minimum of that many independent exponential random variables, we have for every $m$
\begin{equation}\label{eq::splittime_exp}
\tau_m\stackrel{d}{=} \sum_{i=1}^m E_i/S^{(n)}_i,
\end{equation}
with $E_i$ i.i.d. Exp$(1)$. Now let $\xi^{(1)}, \ldots, \xi^{(r)}, \xi_1, \xi_2, \ldots$ denote independent standard normal random variables. Recall that $D_i^{(n)}$ denotes the number of children of the i-th dying particle. Then, by the CLT we have
\begin{equation*}
(D_i^{(n)}|\text{type }s \text{ splits})\toindis \sum_{t\in\CS} \left(\lambda_{st}^{(n)} + \sqrt{\lambda_{st}^{(n)}}\cdot\xi^{(t)}\right) \stackrel{d}{=} \wla_n+1+\sqrt{\wla_n+1}\cdot\xi_i.
\end{equation*}
From here, with the usual notation $S_i^{(n)}=\sum_{j=1}^i D_j^{(n)} - (i-1)$:
\begin{align*}
S_i^{(n)} &\toindis i\wla_n+1+\sqrt{i(\wla_n+1)}\cdot\tilde\xi_i.
\end{align*}
Applying this result to \eqref{eq::splittime_exp} yields $\wla_n \tau_m  \approx \sum_{i=1}^m E_i/i.$

Notice that the sequence $E_m/m, E_{m-1}/m-1, \ldots, E_1/1$ gives in distribution the spacings of the exponential random variables $E_1, \ldots, E_m$. So the sum $\sum_{i=1}^mE_i/i$ is equal in distribution to a random variable $B_m$ that is the maximum of $m$ independent exponentially distributed random variables with rate 1. The distribution function of $B_m$ is
\begin{equation*}
\prob{B_m\leq x} = \left(1-e^{-x}\right)^m.
\end{equation*}
From here we find that
\begin{align*}
\prob{\wla_n \tau^{(x)}_{a_n}-\log a_n\leq x} &= \prob{B_{a_n}\leq x+\log a_n }
= \left(1-e^{-(x+\log a_n)}\right)^{a_n} \\
&= \exp\left(-e^{-x}+O(1/a_n)\right) \to \exp\left(-e^{-x}\right)=\Lambda(x),
\end{align*}
where $\Lambda$ denotes the distribution function of a standard Gumbel random variable. Similarly to the proof in Section \ref{subsec::prooffinite}, we conclude that
\begin{equation*}
(\wla_n \tau^{(x)}_{a_n}-\log a_n,\, \wla_n \tau^{(y)}_{C_n^{\text{con}}}-\log C_n^{\text{con}})\toindis (Y_1, Y_2),
\end{equation*}
where $Y_1$ and $Y_2$ are two independent copies of a standard Gumbel random variable. From Theorem \eqref{eq::minpath} we get that by the continuous mapping theorem
\begin{equation}\label{eq::densepath}\ba
\wla_n \CP_n -\log n &\stackrel{d}{=} \underbrace{\wla_n \tau^{(x)}_{a_n}-\log a_n}_{\toindis Y_1} -\log \frac{n}{a_n}\\+  &\min_i \Big\{  \underbrace{\wla_n \tau^{(y)}_{C_n^{(i)}}-\log C_n^{(i)}}_{\toindis Y_2} + \wla_n\left(\frac{1}{\wla_n}\log C_n^{(i)}+ E_i\right)\Big\}.
\ea
\end{equation}
Now using Proposition \ref{prop::WeakConvOfCn} to see that $C_n^{(i)}=n/a_n P_n^{(i)}$ , with $P_n^{(i)}$ PPP$(\hla_n)$ points, and then Lemma \ref{lemma::gumbeldistribution} yields that the last term in the minimum equals
\[ \wla_n\min_i \left\{\frac{1}{\wla_n}\log C_n^{(i)}+ E_i \right\}= \log \frac{n}{a_n} - X_3 + \log\frac{(\wla_n+1)}{\hla_n}, \]
with $X_3$ a standard Gumbel variable. Then, under Assumption \ref{ass::sorosszeg_la} $\hla_n=\wla_n$, thus the last term vanishes in the limit. Combining this with \eqref{eq::densepath} finishes the proof.

\bibliographystyle{plain}
\bibliography{bibliography}
\addcontentsline{toc}{section}{References}

\end{document}